\documentclass[]{article}

\usepackage{multirow}
\usepackage{longtable}
\usepackage{xcolor}
\usepackage{algorithm, algpseudocode}
\usepackage{geometry}
\usepackage{amsmath, amsfonts, amsthm, amssymb}
\usepackage{url}

\newcommand{\blue}[1]{#1}

\newtheorem{theorem}{Theorem}
\newtheorem{lemma}{Lemma}
\newtheorem{proposition}{Proposition}
\newtheorem{definition}{Definition}
\newtheorem{example}{Example}
\newtheorem{remark}{Remark}

\newcommand{\R}{\mathbb{R}}
\newcommand{\N}{\mathbb{N}}

\newcommand{\X}{\mathbb{X}}
\newcommand{\Y}{\mathbb{Y}}
\newcommand{\Z}{\mathbb{Z}}
\renewcommand{\S}{\mathbb{S}}

\newcommand{\cA}{\mathcal{A}}

\newcommand{\cC}{\mathcal{C}}
\newcommand{\cD}{\mathcal{D}}
\newcommand{\cE}{\mathcal{E}}
\newcommand{\hcE}{\widehat{\mathcal{E}}}
\newcommand{\cF}{\mathcal{F}}

\newcommand{\cH}{\mathcal{H}}
\newcommand{\cI}{\mathcal{I}}
\newcommand{\cJ}{\mathcal{J}}

\newcommand{\cL}{\mathcal{L}}
\newcommand{\cM}{\mathcal{M}}
\newcommand{\cN}{\mathcal{N}}
\newcommand{\cO}{\mathcal{O}}

\newcommand{\cR}{\mathcal{R}}

\newcommand{\cT}{\mathcal{T}}
\newcommand{\cU}{\mathcal{U}}

\def\bDelta{\mbox{\boldmath $\Delta$}}

\newcommand{\inprod}[2]{\left\langle #1, #2 \right\rangle}

\newcommand{\norm}[1]{\left\lVert #1 \right\rVert}
\newcommand{\abs}[1]{\left\lvert #1 \right\rvert}

\newcommand{\st}{\mathrm{s.t.}}

\title{A Squared Smoothing Newton Method for Semidefinite Programming}

\author{Ling Liang \thanks{Department of Mathematics, University of Maryland at College Park, MD 20742, USA, {liang.ling@u.nus.edu}.}, Defeng Sun \thanks{Department of Applied Mathematics, The Hong Kong Polytechnic University, Hung Hom, Hong Kong, \\ {defeng.sun@polyu.edu.hk}.}, Kim-Chuan Toh \thanks{Department of Mathematics and Institute of Operations Research and Analytics, National University of Singapore, Singapore 119076, {mattohkc@nus.edu.sg}.}}
\date{}

\begin{document}
\maketitle

\begin{abstract}
	This paper proposes a squared smoothing Newton method via the Huber smoothing function for solving semidefinite programming problems (SDPs). We first study the fundamental properties of the matrix-valued mapping defined upon the Huber function. Using these results and existing ones in the literature, we then conduct rigorous convergence analysis and establish convergence properties for the proposed algorithm. In particular, we show that the proposed method is well-defined and admits global convergence. Moreover, under suitable regularity conditions, i.e., the primal and dual constraint nondegenerate conditions, the proposed method is shown to have a superlinear convergence rate. To evaluate the practical performance of the algorithm, we conduct extensive numerical experiments for solving various classes of SDPs. Comparison with the state-of-the-art SDP solvers demonstrates that our method is also efficient for computing accurate solutions of SDPs.
\end{abstract}%



\section{Introduction}
\label{intro}

{A standard} primal linear semidefinite programming (SDP) problem {is} the problem of minimizing a linear function in the space $\S^n$ subject to $ m$ linear equality constraints and the essential positive semidefinite constraint $X\in \S_+^n$. Mathematically, the primal SDP has its standard form:
\begin{equation}
	\label{eq-sdp-p}
	\min_{X\in \S^n}\; \inprod{C}{X} \quad\st\quad \cA X = b,\; X \in \S_+^n,
\end{equation}
where $C\in \S^n$, $b\in \R^m$ are given data, and $\cA:\S^n\rightarrow \R^m$ is a linear mapping {given by} $\cA X := \begin{pmatrix}
	\inprod{A_1}{X} &
	\dots &
	\inprod{A_m}{X}
\end{pmatrix}^T$ for all $X\in \S^n$,
with given matrices $A_i \in \S^n$, for $i = 1,\dots, m$, and $\inprod{\cdot}{\cdot}$ denoting the standard trace inner product. Note that the adjoint of $\cA$, denoted by $\cA^*$, is a linear mapping from $\R^m$ to $\S^n$ defined as $\cA^*y := \sum_{i = 1}^m y_iA_i$ for $y\in \R^m$.
Associated with problem \eqref{eq-sdp-p}, the Lagrangian dual {problem} of \eqref{eq-sdp-p} {is given by}
\begin{equation}
	\label{eq-sdp-d}
	\max_{y\in \R^m, Z\in \S^n}\; \inprod{b}{y}\quad \st\quad \cA^* y + Z = C, \; Z\in \S_+^n.	
\end{equation}

In this paper,  we assume that the primal problem \eqref{eq-sdp-p} admits at least one optimal solution and satisfies the Slater's condition, i.e., there exists $\tilde X \in \S^n_{++}$ such that $\cA \tilde X = b$. Under these assumptions, the dual problem \eqref{eq-sdp-d} {has an optimal solution and the dual optimal value is equal to the primal optimal value \footnote{Note, however, that the strong duality may not hold for \eqref{eq-sdp-p}-\eqref{eq-sdp-d} in general (see e.g., \cite{vandenberghe1996semidefinite}).}.} Thus, the duality gap between \eqref{eq-sdp-p} and \eqref{eq-sdp-d} is {zero}. As a consequence, the following system of KKT optimality conditions, given as
\begin{equation}\label{eq-KKT-0}
	\cA X = b, \quad \cA^*y + Z = C,\quad X\in \S_+^n,\; Z\in \S_+^n,\; \inprod{X}{Z} = 0,
\end{equation}
admits a solution. We call an arbitrary triple $(\bar X, \bar y, \bar Z)\in \S^n\times \R^m \times \S^n$ a KKT point if it satisfies the KKT conditions in \eqref{eq-KKT-0}. Let $\bar X$ be an optimal solution to the primal problem \eqref{eq-sdp-p}, for simplicity, we denote the set of {associated}  Lagrangian multipliers $\cM(\bar X)$ as
\begin{equation}\label{eq-set-mult}
	\cM(\bar X) := \left\{(y, Z)\in \R^m\times \S^n\;:\; (\bar X, y, Z) \textrm{ is a KKT point} \right\}.
\end{equation}
Then, under the aforementioned conditions, $\cM(\bar X)$ is a nonempty set.

The research on SDPs has been active for decades and still receives {strong} attention to date. Indeed, SDP has become one of the fundamental modeling and optimization tools {which} encompasses a wide range of applications in different fields. The increasing interest in SDP has resulted in fruitful and impressive works in the literature. For theoretical developments and many important applications of SDP in engineering, finance, optimal control, statistics, machine learning, combinatorics, and beyond, we refer the reader to \cite{bonnans2013perturbation, chan2008constraint, majumdar2020recent, sun2006strong, vandenberghe1996semidefinite, wolkowicz2012handbook}, to mention just a few. In the next few paragraphs, we will briefly review some influential works {on} developing efficient solution methods, including interior point methods (IPMs), first-order methods (FOMs), augmented Lagrangian methods (ALMs), and smoothing Newton methods, for solving SDPs.

Perhaps the most notable and reliable algorithms for solving SDPs are based on the {IPMs}, see \cite{todd2001semidefinite, wolkowicz2012handbook} for comprehensive surveys \footnote{See also the benchmark website \url{https://plato.asu.edu/ftp/sparse_sdp.html} for state-of-the-art IPM-based SDP solvers.}. IPMs admit polynomial worst-case complexity and {are able to deliver} highly accurate solutions. The essential idea of a {primal-dual} IPM is to solve the KKT conditions in \eqref{eq-KKT-0} indirectly by solving a sequence of perturbed KKT conditions:
\[
\cA X = b, \quad \cA^* y + Z = C,\quad XZ = \mu I,\quad \mu > 0,
\]
with $\mu$ being driven to zero. At each iteration of the IPM, one usually needs to solve a linear system whose coefficient matrix (i.e., the Schur complement matrix) is of {size} ${m\times m}$ and generally fully dense {and ill-conditioned}. Thus, IPMs may not be suitable for some practical applications, since SDPs arising from such applications are typically of large sizes. To alleviate these issues, many modifications and extensions of IPMs have been developed. Two main approaches are: (a) investigating iterative linear system solvers for computing the search directions \cite{toh2004solving,toh2002solving} in order to handle large-scale linear systems; (b) exploiting the underlying sparsity structures (such as the chordal structure of the aggregate sparsity pattern) in the SDP data for more efficiency \cite{kim2011exploiting, kovcvara2021decomposition}. For (a), we note that iterative linear system solvers (such as the preconditioned conjugate gradient (PCG) method) require carefully designed preconditioners to handle the issue of ill-conditioning. However, finding an effective preconditioner is still a challenging task. For (b), we note that not all SDPs admit appealing sparsity structures, and hence the range of applications for this approach remains limited. In summary, there exist several critical issues that may stop {one from applying an} {IPM} and its variants for solving large-scale \footnote{\blue{In the sense that $n$ is several thousands and $m$ is several hundreds of thousands.}} SDPs in real-world applications.

FOMs are at the forefront {on} developing scalable algorithms for solving large-scale optimization problems due to the low per-iteration complexity. For SDPs, we mention some {relevant} FOMs to capture the picture {on} recent developments along this direction. In the early \blue{21st} century, \cite{helmberg2000spectral} applied a spectral bundle method {to solve} a special regularized problem for the dual SDP \eqref{eq-sdp-d}. The authors proved the convergence of the proposed method under the condition that the trace of any primal optimal solution is fixed. FOMs such as the (accelerated) proximal gradient method are also popular; See for instance the {works} \cite{jiang2012inexact,yang2021stride} together with some numerical evaluations of their practical performance. Later,  \cite{renegar2014efficient} proposed an FOM which transforms the SDP into an equivalent convex optimization problem. But no numerical {experiment was} conducted in this work, and hence the practical performance of the proposed algorithm is unclear. Meanwhile, operator splitting methods and the alternating direction method of multipliers (ADMM) and their variants \cite{povh2006boundarypoint, sun2015convergent, zhao2010newton, zheng2020chordal} have been demonstrated to be well suited for solving large-scale SDPs, {although high accuracy solutions may not be achievable in general.} More recently, in order to design storage-efficient algorithms, \cite{yurtsever2021scalable} proposed a conditional-gradient-based augmented Lagrangian method equipped with random sketching techniques for solving SDPs with {fixed} trace constraints. {The method} has state-of-the-art performance in memory consumption, but whether it is also efficient in computational time for obtaining {high accuracy} solutions of general SDPs (with or without {fixed} trace constraints) {requires} further investigations. Another similar storage-optimal framework can be found in \cite{ding2021optimal}. {For sparse SDPs, \cite{jiang2022bregman} recently proposed a Bregman distance-based Chambolle–Pock primal–dual algorithm (note that the classical Chambolle-Pock algorithm is an equivalent form of the so-called linearized ADMM \cite{chambolle2011first}) with provable convergence properties which exploits data sparsity to reduce the computational cost for evaluating the Bregman proximal operator.} Indeed, those existing works have shown that FOMs are scalable algorithms.  However, a commonly accepted fact for FOMs is that they may not have favorable efficiency for computing accurate solutions {or they may even fail to deliver moderately accurate solutions}. Therefore, when {high quality} solutions are {needed}, FOMs may not be attractive.

Being studied for decades, augmented Lagrangian methods (ALMs) have been shown to be suitable for solving large-scale optimization problems (including SDPs) efficiently and accurately. Many efficient algorithms based on the ALM have been proposed. For example,  \cite{jarre2008augmented} developed an augmented primal-dual method. Later, \cite{malick2009regularization} proposed a Moreau-Yosida regularization method for the primal SDP \eqref{eq-sdp-p}. Both methods perform reasonably well on some SDPs with a relatively large number of linear constraints.  \cite{zhao2010newton} {developed} a dual-based ALM (i.e. the ALM is applied to the dual problem \eqref{eq-sdp-d}) whose {design principle is} fundamentally different from the one in \cite{malick2009regularization} by {relying on} the deep connection between the primal proximal point algorithm and the dual-based ALM \cite{rockafellar1976augmented,rockafellar1976monotone}, {and the highly efficient semi-smooth Newton method \cite{qi1993nonsmooth}}. The solver developed in \cite{zhao2010newton} has shown very promising practical performance for a large {collection} of SDPs compared with existing ones. Moreover, practical experience shows that when a good initial point is available, the performance of the algorithm can be further improved. This has motivated the same group of researchers to develop a hybrid framework, namely {\tt SDPNAL+} \cite{yang2015sdpnal}, which combines the dual-based ALM \cite{zhao2010newton} and the ADMM method \cite{sun2015convergent} with some majorization techniques.  {\tt SDPNAL+} has demonstrated excellent numerical performance, particularly when solving SDPs with large number of constraints (say more than 20,000), and it can handle additional polyhedral constraints. Lastly, we mention another popular ALM-based approach for solving low-rank SDPs, namely, the Burer-Monteiro (BM) method \cite{burer2003nonlinear}. The key ingredient of the BM method is to replace the essential constraint $X\in \S_+^n$ with the low-rank factorization $X = RR^T$ {which reformulates \eqref{eq-sdp-p} as a nonconvex optimization problem in $ \R^{n\times r} $, where $ r $ is a prior bound on the rank of an optimal solution}. \blue{Following the same research theme of \cite{burer2003nonlinear}, recent works \cite{wang2023decomposition, wang2023solving} demonstrate that one can expect more promising practical performance for solving low-rank SDPs via incorporating the matrix factorization technique with manifold optimization. Moreover, many} important works on the connection between the factorized problem and the primal SDP \eqref{eq-sdp-p} have also been published; see for instance \cite{boumal2016non, burer2005local, pataki1998rank}. However, choosing a suitable $ r $ to balance a conservative (large) and an aggressive (small) choice is a delicate task as the former will lead to higher computational cost per iteration while the latter may lead to convergence failure. As a consequence, the practical performance of the BM method is sensitive to the choice of $ r $.

Different from IPMs, one could also consider solving the nonsmooth reformulation of the KKT conditions in \eqref{eq-KKT-0}:
\begin{equation}
	\label{eq-KKT}
	\cF(X,y,Z) :=
	\begin{pmatrix}
		\cA X - b \\
		-\cA^* y - Z + C \\
		X - \Pi_{\S_+^n}(X-Z)
	\end{pmatrix} = 0,\quad (X,y, Z)\in \S^n\times \R^m\times \S^n,	
\end{equation}
where {the} conditions $X\in \S_+^n,\; Z\in \S_+^n,\; \inprod{X}{Z} = 0$ are replaced by a single nonsmooth equation $X - \Pi_{\S_+^n}(X-Z) = 0$ and $\Pi_{\S_+^n}(\cdot)$ is the projection operator onto $\S_+^n$. Obviously, the function $\cF(\cdot)$ is nonsmooth due to the nonsmoothness of $\Pi_{\S_+^n}(\cdot)$. Thus, the {classical} Jacobian of $\cF(\cdot)$ is not well-defined and the classical Newton method is not applicable for solving \eqref{eq-KKT}.  Since the {B-subdifferential} and the Clarke's generalized Jacobian \cite{clarke1990optimization} of $\Pi_{\S_+^n}(\cdot)$, namely, $ \partial_B \Pi_{\S_+^n} $ and $ \partial \Pi_{\S_+^n} $, are well-defined, and $\Pi_{\S_+^n}(\cdot)$ is strongly semi-smooth \cite{sun2002semismooth}, $ \partial_B\cF $ and $ \partial \cF $ are both well-defined, and $ \cF(\cdot) $ is also strongly semismooth. Then, it is natural to apply the semi-smooth Newton (SSN) method \cite{qi1993nonsmooth} that has a {fast local} convergence rate under suitable regularity conditions. However, the global convergence of the SSN method {needs} a valid merit function {for which} the line search procedure for computing a step size is well-defined. Typically, such a valid merit function should be {continuously} differentiable or satisfy other stronger conditions. Yet, defining such a merit function from $\cF(\cdot)$ directly turns out to be {challenging}. This motivates us to {develop a smoothing Newton method since a suitable merit function can readily be obtained, and its} global convergence can be established. Moreover, the smoothing Newton method inherits the strong local convergence properties from the classical Newton method and the semismooth Newton method, under similar regularity conditions.

Existing smoothing Newton methods have been developed and studied extensively in different areas. To mention just a few of these works, the reader is referred to \cite{kanzow1999jacobian, chen1998global, qi1995globally, gao2010calibrating, qi2000new, chan2008constraint, chen2003analysis, chen2003non}. These smoothing methods can be roughly divided into two groups, Jacobian smoothing Newton methods \cite{kanzow1999jacobian,chen1998global} and squared smoothing Newton methods \cite{gao2010calibrating, qi2000new, chan2008constraint}. The convergence analysis of Jacobian smoothing Newton methods strongly depends on the {so-called} Jacobian consistency property and other strong conditions for the underlying smoothing functions. However, many smoothing functions, such as those functions defined via normal maps \cite{robinson1992normal}, do not satisfy these conditions. Furthermore, to get stronger {convergence} results, it is often useful to add a small perturbation term to the smoothing function. In this case, conditions ensuring those stronger results for the Jacobian smoothing Newton methods are generally not {satisfied}.

Smoothing Newton methods rely on an appropriate smoothing function for the plus function
\[
\rho(t) := \max\{0, t\},\quad t\in\R,
\]
in order to deal with the nonsmooth projector $\Pi_{\S_+^n}(\cdot)$. To the best of our knowledge, the most notable and commonly used smoothing function {is the} so-called Chen-Harker-Kanzow-Smale (CHKS) function \cite{chen1993non, kanzow1996some, smale2000algorithms}, 
\[
    \xi(\epsilon, t) = (\sqrt{t^2+4\epsilon^2}+t)/2,\quad  (\epsilon, t)\in \R\times \R ,
\]
which has been extensively studied and used for solving SDPs and semidefinite linear and nonlinear complementarity problems. For other smoothing functions whose properties have also been well-studied, the reader is referred to \cite{sun1999ncp, qi2000new} for more details. However, it is easy to see that $\xi(\epsilon, t)$ maps any negative number $t$ to a positive one when $\epsilon\neq 0$, {even though $\rho(t) =0$.} Thus, it destroys the possible sparsity structure when evaluating the Jacobian of the merit function. Hence, smoothing Newton methods based on the CHKS function would require more computational effort {when performing the matrix-vector multiplications needed in an iterative linear system solver for computing the search directions; see Section \ref{subsection-implementaion} for more detailed discussions.} To resolve this issue, we propose to use the Huber smoothing function \cite{pinar1994smoothing} which maps any negative number to zero so that the underlying sparsity structure of the plus function is inherited. To use the Huber smoothing function, we then need to study some fundamental properties, including continuity, differentiability and (strong) semismoothness of the matrix-valued mapping associated with the Huber function.

Although {much progress {has} been made in developing efficient algorithms for solving SDPs in the literature}, the efficiency and cost of these methods are {still} far from satisfactory. This motivates us to study and analyze the squared smoothing Newton method {via} the Huber smoothing function. To evaluate the efficiency, we implement the algorithm and conduct extensive numerical experiments by solving several classes of SDPs. Our theoretical analysis and numerical results show that the proposed method admits {appealing} convergence properties and {it has shown promising numerical performance} for computing {relatively} accurate solutions.  {Compared to the ALM based method in \cite{zhao2010newton} for solving an SDP problem, our smoothing Newton method has the {advantage that it can also handle a relatively large number of linear constraints and has the} {feature} that it is applied to a single nonlinear system of equations, whereas the ALM method needs to solve a sequence of subproblems for which the total number of Newton {iterations} {is usually} more than the former.} Moreover, since the practical performance of the smoothing Newton method for solving SDPs {has not been systematically evaluated} in the literature, our theoretical and numerical studies can certainly help one to gain a better understanding of this class of algorithms. {However, note that a comprehensive comparison of the practical performance of different smoothing functions is
{not the main focus of this research.}}

The rest of the paper is organized as follows. In {Section} \ref{preliminary}, we {summarize} some existing results that will be used {in later} analysis. Then, we study the continuity, differentiability and semismoothness of the {matrix-valued} mapping defined upon the Huber function in Section \ref{huber}. Using these results, we are able to analyze the correctness, global convergence and local fast convergence rate of the proposed {squared smoothing Newton} method in Section \ref{smooth-newton}. In Section \ref{numerical-exp}, we conduct numerical experiments for solving various classes of SDPs and evaluate the practical performance of the proposed method. Finally, we conclude {the} paper in Section \ref{conclusions}.

\section{Preliminary}\label{preliminary}

Let $ \Y $ and $\Z$ be any two finite dimensional real vector spaces, each equipped with {an inner} product $\inprod{\cdot}{\cdot}$ and its induced norm $ \norm{\cdot} $. \blue{Assume that $\cO \subseteq \Y$ is an open set and $\Theta:\cO\rightarrow \Z$ is a function on $\cO$. If $\Theta$ is locally Lipschitz continuous on $\mathcal{O}$, then it is Fr\'{e}chet-differentiable almost everywhere on $\cO$ \cite{rockafellar2009variational}, i.e., for almost every $y\in \mathcal{O}$,  there exist a linear operator depending on $y$, denoted as $\Theta'(y):\Y \to \Z$, such that
\[
    \lim_{\norm{\Delta y}\to 0}\frac{\norm{\Theta(y + \Delta y) - \Theta(y) - \Theta'(y)\Delta y}}{\norm{\Delta y}} = 0,
\]
and $\Theta'(y)$ is called the F-derivative of $\Theta$ at $y\in\mathcal{O}$. Hence, the B-subdifferential \cite{qi1993convergence} of $\Theta$ at $y\in \cO$, denoted by $\partial_B\Theta(y)$, is well-defined and given by
\[
\partial_B\Theta(y) := \left\{V:V = \lim_{k\rightarrow \infty} \Theta'(y^k), y^k\rightarrow y, {\Theta'(y^k)} \textrm{ is {the F-derivative} of $\Theta$ at } y^k\right\}.
\]
Moreover, Clarke's generalized Jacobian \cite{clarke1990optimization} of $\Theta$ at  $y\in \cO$ is defined as $\partial \Theta(y) := \mathrm{conv}\left\{\partial_B\Theta(y)\right\}$, where ``$ \mathrm{conv} $'' denotes the convex hull of the underlying set. From \cite{clarke1990optimization}, we know that $\partial_B\Theta(\cdot)$ is nonempty and upper semi-continuous for all $y\in\mathcal{O}$. Furthermore, if $\Theta'(y)$ exists for a certain $y\in\mathcal{O}$, then $\Theta'(y)\in \partial_B\Theta(y)$. The above statements hold true if one replaces $\partial_B\Theta(y)$ with $\partial\Theta(y)$.
}

\subsection{Properties of $\Pi_{\S_+^n}(\cdot)$}
Denote the matrix-valued mapping $\Omega_0: \R^{N}\rightarrow \S^N$ for any integer $N>0$ as
\begin{equation}
	\label{eq-first-diff-max}
	[\Omega_0(d)]_{ij} = \frac{\rho(d_i) + \rho(d_j)}{\abs{d_i}+\abs{d_j}},\quad i,j = 1,\dots, N,\quad \forall d\in \R^N,
\end{equation}
{with the convention that $0/0 := 1$}. Let the spectral decomposition of $W \in \S^n$ be
\begin{equation}\label{eq-eig-W}
	W = PDP^T,\quad D = \mathrm{diag}(d),\quad d = (d_1,\dots, d_n)^T\in \R^n,\quad  d_1\geq \dots \geq d_n,
\end{equation}
where $P\in \R^{n\times n}$ is orthogonal. For later usage, we also define {the following} three index sets: $\alpha := \{i : d_i > 0\}$, $\beta := \{i : d_i = 0\}$, and $\gamma := \{i : d_i < 0\}$,
{corresponding to the} positive, zero and negative eigenvalues of $W$, respectively.

It is clear that both $\partial_B \Pi_{\S_+^n}(W)$ and $\partial \Pi_{\S_+^n}(W)$ are well-defined for any $ W\in \S^n $ (since $\Pi_{\S_+^n}(\cdot)$ is {globally} Lipschitz continuous with modulus $ 1 $ on $\S^n$). We then have the following useful lemma {(see \eqref{eq-V0-VH})} on $\partial_B \Pi_{\S_+^n}(W)$ and $\partial \Pi_{\S_+^n}(W)$ whose proof can be found in \cite[Proposition 2.2]{sun2006strong}. {Here we use ``$\circ$'' to denote the Hadamard product.}
\begin{lemma}
	\label{lemma-partial-Pi}
	Suppose that $W\in \S^n$ has the spectral decomposition \eqref{eq-eig-W}. Then $V\in \partial_B \Pi_{\S_+^n}(W)$ (respectively, $\partial \Pi_{\S_+^n}(W)$) if and only if there exists $ V_{\abs{\beta}}\in \partial_B\Pi_{\S_+^{\abs{\beta}}}(0) $ (respectively, $ \partial \Pi_{\S_+^{\abs{\beta}}}(0) $) such that
	\[
	V(H) = P
	\begin{pmatrix}
		\tilde H_{\alpha\alpha} & \tilde H_{\alpha \beta} & [\Omega_0(d)]_{\alpha\gamma}\circ \tilde H_{\alpha \gamma} \\
		\tilde H_{\alpha \beta}^T & V_{\abs{\beta}}(\tilde H_{\beta\beta}) & 0 \\
		H_{\alpha \gamma}^T\circ [\Omega_0(d)]_{\alpha\gamma}^T & 0 & 0
	\end{pmatrix}
	P^T,
	\]
	where $\tilde H := P^THP$ for any $H\in \S^n$. In particular, $V_{\abs{\beta}}\in \partial_B\Pi_{\S_+^{\abs{\beta}}}(0)$ if and only if there exist an orthogonal matrix $U\in \R^{\abs{\beta}\times \abs{\beta}}$ and
	\[
	\Omega_{\abs{\beta}} \in \left\{\Omega \in \S^{\abs{\beta}}: \Omega = \lim_{k\rightarrow \infty} \Omega_0(z^k),\;\R^{\abs{\beta}}\ni z^k\rightarrow 0,\; z^k_1\geq \dots \geq z^k_{\abs{\beta}},\; z^k\neq 0 \right\}
	\]
	such that
	\[
	V_{\abs{\beta}}(Y) = U\left[ \Omega_{\abs{\beta}}\circ \left(U^TYU\right)\right]U^T,\quad \forall Y\in \S^{\abs{\beta}}.
	\]
\end{lemma}

Another critical and fundamental concept of $\Pi_{\S_+^n}(\cdot)$ is its semismoothness, which plays an essential role in analyzing the convergence of Newton-type algorithms. Recall that a mapping $\Theta:\Y\rightarrow \Z$ is said to be directionally differentiable at the point $y\in \Y$ if the limit, defined by
\[
\Theta(y; h):=\lim_{t\downarrow0} \frac{\Theta(y + th) - \Theta(y)}{t},
\]
exists for any $h\in \Y$. Then the definition of the semismoothness is given as follows.

\begin{definition}
	\label{def-semismooth}
	Let $\Theta:\cO\subseteq \Y\rightarrow \Z$ be a locally Lipschitz continuous function. $\Theta$ is said to be semismooth at $y\in \cO$ if $\Theta$ is directionally differentiable at $y$ and for any $V\in \partial \Theta(y+h)$,
	\[
	\norm{\Theta(y + h) - \Theta(y) - Vh} = o(\norm{h}),\quad h\rightarrow 0.
	\]
	$\Theta$ is said to be strongly semismooth at $y\in \Y$ if $\Theta$ is semismooth {at $y$} and for any $V\in \partial \Theta(y+h)$
	\[
	\norm{\Theta(y + h) - \Theta(y) - Vh} = O(\norm{h}^2),\quad h\rightarrow 0.
	\]
	We say {that} $\Theta$ is semismooth (respectively, strongly semismooth) if $\Theta$ is semismooth (respectively, strongly semismooth) at every $y\in \Y$.
\end{definition}
Semismoothness was originally introduced in \cite{mifflin1977semismooth} for functionals. \cite{qi1993nonsmooth} extended it {to vector-valued} functions. It is also well-known that the projector $\Pi_{\S_+^n}(\cdot)$ is strongly semismooth on $\S^n$ \cite{sun2002semismooth}. Later, we will also show that the smoothed counterpart of $\Pi_{\S_+^n}(\cdot)$ is also strongly semismooth {on $ \S^n $} (see Proposition \ref{prop-huber-properties}).

\subsection{Equivalent conditions}
For the rest of this section, we describe several important conditions related to {the theoretical analysis for} SDPs and present {some connections} between these conditions. The presented materials are borrowed directly from the literature; see for instance \cite{chan2008constraint} and the references therein.

First, we recall the general concept of constraint nondegeneracy. Let $g:\Y \rightarrow \Z$ be a continuously differentiable function and $\cC$ be a {nonempty closed} convex set in $\Z$. Consider the following {feasibility} problem:
\begin{equation}\label{eq-feasibility-prob}
	g(\xi) \in \cC,\quad \xi\in \Y.
\end{equation}
Let $\bar \xi \in \Y$ be a feasible solution to the above {feasibility} problem. The tangent cone of $\cC$  at the point $g(\bar \xi )$ is denoted by $\cT_\cC(g(\bar \xi ))$, and the largest linear subspace of $\Y$ contained in $\cT_\cC(g(\bar \xi ))$, (i.e., the linearity space of $\cT_\cC(g(\bar \xi ))$) is denoted by $\mathrm{lin}(\cT_\cC(g(\bar \xi )))$.
\begin{definition}\label{def-non}
	A feasible solution $\bar \xi$ for \eqref{eq-feasibility-prob} is constraint nondegenerate if
	\[
	g'(\bar \xi) \Y + \mathrm{lin}(\cT_\cC(g(\bar \xi))) = \Z.
	\]
\end{definition}
The concept of degeneracy was introduced by Robinson \cite{robinson1983local, robinson1984local, robinson1987local} and the {term} ``constraint nondegeneracy'' was also coined by Robinson \cite{robinson2003constraint}. For nonlinear programming (i.e., $\Y$ is the Euclidean space $\R^m$ and $\cC = \{0\}^{m_1}\times \R_+^{m_2}$ with $m = m_1+m_2$), the constraint {nondegeneracy} condition is equivalent to the {well known} linear independence constraint qualification (LICQ) \cite{robinson1984local}.

Let $\cI$ denote the identity map on $\S^n$. Applying Definition \ref{def-non}, we see that the primal constraint nondegeneracy holds at a feasible solution $\bar X\in \S_+^n$ {of} the primal problem \eqref{eq-sdp-p} if
\begin{equation}\label{eq-def-primal-non}
	\begin{pmatrix}
		\cA \\ \cI
	\end{pmatrix} \S^n +
	\begin{pmatrix}
		\{0\} \\
		\mathrm{lin}(\cT_{\S_+^n}(\bar X))
	\end{pmatrix} =
	\begin{pmatrix}
		\R^m \\ \S^n
	\end{pmatrix}.
\end{equation}
Similarly, the dual constraint nondegeneracy holds at a feasible solution $(\bar y, \bar Z)\in \R^n \times \S_+^n$ {of} the dual problem \eqref{eq-sdp-d} if
\begin{equation}\label{eq-def-dual-non}
	\begin{pmatrix}
		\cA^* & \cI \\ 0 & \cI
	\end{pmatrix}
	\begin{pmatrix}
		\R^m \\ \S^n
	\end{pmatrix} +
	\begin{pmatrix}
		\{0\} \\ \mathrm{lin}(\cT_{\S_+^n}(\bar Z))
	\end{pmatrix} =
	\begin{pmatrix}
		\S^n \\ \S^n
	\end{pmatrix}.
\end{equation}
Without much difficulty, one can show that the primal constraint nondegeneracy \eqref{eq-def-primal-non} is equivalent to $ \cA \mathrm{lin}(\cT_{\S_+^n}(\bar X)) = \R^m $, and the dual constraint nondegeneracy \eqref{eq-def-dual-non} is equivalent to $ \cA^*\R^m + \mathrm{lin}(\cT_{\S_+^n}(\bar Z))  = \S^n $. Moreover, {if} $\bar X$ {is} an optimal solution for \eqref{eq-sdp-p}, {then} under the primal constraint {nondegeneracy} condition, we can show that $\cM(\bar X)$ is a singleton {\cite[Theorem 2]{alizadeh1997complementarity}}.

It is shown in \cite[Theorem 18]{chan2008constraint} that the primal and dual constraint {nondegeneracy} conditions are both necessary and sufficient conditions for the nonsingularity of elements in {both} $\partial_B \cF(\bar X, \bar y, \bar Z)$ and $\partial \cF(\bar X, \bar y, \bar Z)$, where $ (\bar X, \bar y, \bar Z) $ is a KKT point. {Specifically, we have the following theorem.}
\begin{theorem}[\cite{chan2008constraint}]
	\label{thm-pd-nonsingular}
	Let $(\bar X, \bar y, \bar Z) \in \S^n \times \R^m \times \S^n$ be a KKT point. Then the following statements are equivalent:
	\begin{enumerate}
		\item The primal constraint {nondegeneracy} condition holds at $\bar X$ and the dual constraint {nondegeneracy} condition holds at $(\bar y,\bar Z)$, respectively.
		\item Every element in $\partial \cF(\bar X, \bar y, \bar Z)$ is nonsingular.
		\item Every element in $\partial_B \cF(\bar X, \bar y, \bar Z)$ is nonsingular.
	\end{enumerate}
\end{theorem}

We next present the concept {of strong} second-order sufficient condition for linear SDPs.
{We note that the concept was introduced} by \cite{sun2006strong} for general nonlinear SDPs. To this end, let $ (\bar X, \bar y, \bar Z) $ be any KKT point. Write $W = \bar X - \bar Z$, and  assume that $W$ has the spectral decomposition \eqref{eq-eig-W}. Partition $P$ {accordingly with respect to the index set $\alpha, \beta$ and $\gamma$} as $P = {\begin{pmatrix}
		P_\alpha & P_\beta & P_\gamma
\end{pmatrix}}$, {where $P_\alpha\in \R^{n\times \abs{\alpha}}$, $ P_\beta\in \R^{n\times \abs{\beta}} $, and $ P_\gamma\in \R^{n\times \abs{\gamma}} $}. Then, by the fact that $\bar X$ and $\bar Z$ commute, it holds that
\[
\bar X =  P_{\alpha} \mathrm{diag}(d_{1},\dots, d_{\abs{\alpha}}) P_{\alpha}^T,\quad \bar Z = -P_{\gamma} \mathrm{diag}(d_{1+\abs{\alpha}+\abs{\beta}}, \dots, d_{n}) P_{\gamma}^T.
\]

For any given matrix $B\in \S^n$, {let $B^\dagger$ be} the Moore-Penrose pseudoinverse of $B$. We define the linear-quadratic function $\Gamma_B:\S^n\times \S^n\rightarrow \R$ by
\begin{equation}\label{eq-Gamma-B}
	\Gamma_B(S,H) := 2\inprod{S}{HB^\dagger H},\quad \forall (S, H)\in \S^n\times \S^n.
\end{equation}
Using \blue{the above} notation, {the definition of the strong second-order sufficient condition
is given as follows.}

\begin{definition}
	\label{def-sosc}
	Let $\bar X$ be an optimal solution to the primal problem \eqref{eq-sdp-p}. We say that the strong {second-order sufficient condition} holds at $\bar X$ if $ \sup_{(y,Z)\in \cM(\bar X)} \left\{ - \Gamma_{\bar X}(-Z, H)\right\} > 0 $ for any {$ 0\neq H\in \bigcap_{(y,Z)\in \cM(\bar X)} \mathrm{app}(y,Z) $}, where $ \mathrm{app}(y,{Z}) := \left\{B\in \S^n\;:\; \cA B = 0,\; P_{\beta}^TB P_{\gamma} = 0,\; P_{\gamma}^TB P_{\gamma} = 0 \right\}$.
\end{definition}
Finally, we {present} a result that links the strong {second-order sufficient condition} and the dual constraint nondegeneracy {condition}; see \cite[Proposition 15]{chan2008constraint} for a proof.
\begin{lemma}\label{lemma-sosc-dual-non}
	Let $(\bar X, \bar y, \bar Z) \in \S^n\times \R^m \times \S^n$ be a KKT point such that $\cM(\bar X) = \left\{(\bar y, \bar Z)\right\}$. Then, the following two statements are equivalent:
	\begin{enumerate}
		\item The strong {second-order sufficient condition} holds at $\bar X$.
		\item The dual constraint {nondegeneracy} condition holds at $(\bar y, \bar Z)$.
	\end{enumerate}
\end{lemma}

\section{The Huber smoothing function}
\label{huber}

Since the plus function {$\rho(t) =\max\{0, t\}$,} $t\in \R$ is not differentiable at $t = 0$, we consider its Huber smoothing (or approximation) function that is defined as follows:
\begin{equation}\label{eq-huber}
	h(\epsilon, t) = \left\{
	\begin{array}{ll}
		t - \frac{\abs{\epsilon}}{2}, & t > \abs{\epsilon}, \\
		\frac{t^2}{2\abs{\epsilon}}, & 0 \leq t \leq \abs{\epsilon}, \\
		0, & t < 0,
	\end{array}
	\right.\quad \forall (\epsilon, t) \in \R\backslash\{0\}\times \R,\quad h(0, t) = \rho(t), \quad \forall t\in \R.
\end{equation}
Clearly, $ h(\epsilon, t) $ is continuously differentiable for $\epsilon \neq 0$ and $ t\in \R $. Moreover, one can easily check that $h$ is directionally differentiable at $ (0, t) $ for any $ t\in\R $. Since $ \Pi_{\S_+^n}(W) = P {\mathrm{diag}(\rho(d_1), \dots, \rho(d_n))} P^T $, we can compute $\Pi_{\S_+^n}(W)$ approximately by evaluating the {matrix-valued mapping} $\Phi(\epsilon, W)$ that is defined as
\[
\Phi(\epsilon, W):= P\mathrm{diag}(h(\epsilon, d_1), \dots, h(\epsilon, d_n))P^T,\quad \forall (\epsilon, W)\in \R\times \S^n.
\]
In this section, we shall study {some} fundamental properties of  $\Phi$. We note that the techniques used in our analysis are not new {but} mainly borrowed from those in the literature. However, existing techniques were mainly used for analyzing the CHKS-smoothing function. We will show in this section that the same analysis framework is applicable to the Huber smoothing function {\eqref{eq-huber}}.

Before presenting our results, we need to introduce some useful notation. For any $(\epsilon, d)\in \R\backslash\{0\}\times \R^N$, where $ N>0 $ is any dimension, we define the matrix-valued mappings $\Omega : \R\times \R^N\rightarrow \S^N$ and $ \cD: \R\times \R^N\rightarrow \S^N $ as follows:
\begin{equation}
	\begin{aligned}
		\left[\Omega(\epsilon, d)\right]_{ij} := 
		\left\{
		\begin{array}{ll}
			\frac{h(\epsilon, d_i) - h(\epsilon, d_j)}{d_i - d_j} & d_i\neq d_j \\
			h'_2(\epsilon, d_i) & d_i = d_j
		\end{array}
		\right.,\quad 1\leq i,j\leq N, \quad 
		\cD(\epsilon, d) :=  \mathrm{diag}(h_1'(\epsilon, d_1), \dots, h_1'(\epsilon, d_N)).
	\end{aligned}
\label{eq-first-diff}
\end{equation}
Here $ h'_1 $ and $ h'_2 $ denote the partial derivatives with respect to the first and the second arguments of $ h $, respectively. {Note that $ 0\leq [\Omega(\epsilon,d)]_{ij} \leq 1 $, $ 1\leq i, j\leq N $, for any $ (\epsilon, d)\in \R\times \R^N $, and that $ 0 \leq \abs{h_1'(\epsilon, d_i)} \leq \frac{1}{2} $, $ 1\leq i\leq N $, for any $ \epsilon\neq 0 $.}

Let $W$ have the spectral decomposition \eqref{eq-eig-W} and $\lambda_1>\dots>\lambda_r$ be the distinct eigenvalues of $W$ with multiplicities $m_1,\dots, m_r$. Define $s_1 := 0$, $s_j := \sum_{i = 1}^{j-1}m_j$ for $j = 2,\dots,r$, and $s_{r+1} := n$. For $j = 1,\dots, r$, denote the matrices $P_j {:=} \begin{pmatrix}
	p_{s_j+1} & \dots & p_{s_{j+1}} \end{pmatrix} \in \R^{n\times m_j}$ where $p_i$ denotes the $i$-th column of $P$ for $i = 1,\dots, n$, and $ Q_j {:=} P_jP_j^T\in \S^n $. Then, it is clear that
\[
W = \sum_{j = 1}^r \lambda_jQ_j,\quad \Phi(\epsilon, W) = \sum_{j = 1}^r h(\epsilon, \lambda_j)Q_j.
\]
Now, consider $W + tH$, which admits a decomposition $ W + tH = \sum_{i = 1}^n d_i(t)p_i(t)p_i(t)^T $ {with $ d_1(t)\geq \dots \geq d_n(t) $ and $ \{p_1(t),\dots,p_n(t)\} $ forming an orthonormal basis for $ \R^n $}. Similarly, for $j = 1,\dots, r$, we can define the
 matrices $P_j(t) {:=} \begin{pmatrix}
	p_{s_j+1}(t) & \dots & p_{s_{j+1}}(t) \end{pmatrix} \in \R^{n\times m_j} $ and $ Q_j(t) {:=} P_j(t)P_j(t)^T\in \S^n $. By the definition {of directional} differential, we may denote
\[
d_i'(W;H) := \lim_{t\downarrow 0} \frac{d_i(t) - d_i}{t},\quad \forall H\in \S^n,
\]
if the limit exists, for any $i = 1,\dots, n$.

\begin{proposition}\label{prop-huber-properties}
	Given any  $ W \in \S^n$ having the spectral decomposition \eqref{eq-eig-W}. The following hold:
	\begin{enumerate}
		\item For any $ \epsilon\neq 0 $, $ \Phi $ is continuously differentiable at $ (\epsilon, W) $, and it holds that
		\begin{equation}\label{lemma-contin-diff-1}
			\Phi'(\epsilon, W)(\tau, H) 
				 =  P\left[ \Omega(\epsilon, d) \circ (P^THP) + \tau \cD(\epsilon, d)\right]P^T,\quad \forall (\tau, H)\in \R\times \S^n,
		\end{equation}
		{where $ \Omega(\epsilon, d) $, $ \cD(\epsilon, d) $ are defined in \eqref{eq-first-diff}.}
		\item $\Phi$ is locally Lipschitz continuous on $\R\times \S^n$.
		\item $ \Phi $ is directionally differentiable at $ (0,W) $, and {for any $ (\tau, H)\in \R\times \S^n $}, it holds that
		\begin{equation}\label{eq-dir-diff-1}
			\begin{aligned}
				\Phi'((0, W); (\tau, H)) 
				= &\;   \frac{1}{2} \sum_{1\leq k\neq j\leq r} \frac{h(0, \lambda_k)-h(0, \lambda_j)}{\lambda_k- \lambda_j}\left(Q_jHQ_k + Q_kHQ_j\right) \\
				&\; +\sum_{i\in \alpha} \left( d_i'(W; H) - \frac{\abs{\tau}}{2}\right)p_ip_i^T + \sum_{i\in \beta} h(\tau, d_i'(W;H))p_ip_i^T,
			\end{aligned}
		\end{equation}
		where for any $1\leq i \leq n$ with $d_i = \lambda_j$ for some $1\leq j \leq r$, {the} elements of $\{d_i'(W;H): i = s_j+1,\dots, s_{j+1}\}$ are the eigenvalues of $P_j^THP_j$, arranged {in decreasing} order.
		\item $ \Phi $ is strongly semismooth {on $\R\times \S^n $}.
	\end{enumerate}
\end{proposition}

\begin{proof}
	See Appendix \ref{proof-huber-properties}. Interested readers are referred to \cite{gao2010calibrating, zhao2004smoothing} for similar results for other popular smoothing functions. 
\end{proof}

Notice that the matrix $P_j$ is defined {up to} an orthogonal {transformation}, i.e., it can be replaced with $P_jU$, where $U\in \R^{m_j\times m_j}$ is an {arbitrary} orthogonal matrix. For any given $H\in \S^n$, consider the matrix $P_j^THP_j$ which admits the following spectral decomposition
\[
P_j^THP_j = U_j\mathrm{diag}(\mu_1, \dots, \mu_{m_j})U_j^T,\quad \mu_1\geq \dots \geq \mu_{m_j},
\]
where $U_j\in \R^{m_j\times m_j}$ is orthogonal. Then, we may replace $P_j$ by $P_jU_j$. In this way, $P_j^THP_j$ is always a diagonal matrix whose diagonal entries are arranged in decreasing order. As a consequence, we can easily verify from the third statement of Proposition \ref{prop-huber-properties} that
\begin{equation}
	\label{eq-dir-diff-new}
	\begin{aligned}
		\Phi'((0, W);(\tau, H))
		=  &\;
		P \begin{pmatrix}
			\tilde H_{\alpha\alpha} & \tilde H_{\alpha\beta} & [\Omega_0(d)]_{\alpha\gamma}\circ \tilde H_{\alpha\gamma} \\
			\tilde H_{\alpha\beta}^T & \Phi_{\abs{\beta}}(\tau, \tilde H_{\beta\beta}) & 0 \\
			[\Omega_0(d)]_{\alpha\gamma}^T \circ H_{\alpha\gamma}^T & 0 & 0
		\end{pmatrix} P^T   - \frac{\abs{\tau}}{2}\sum_{i\in\alpha}p_ip_i^T,	
	\end{aligned}
\end{equation}
$\forall(\tau, H)\in \R\times \S^n$, where $\tilde H := P^THP$ and the mapping $\Phi_{\abs{\beta}}:\R\times \S^{\abs{\beta}} \rightarrow \S^{\abs{\beta}}$ is defined by replacing the dimension $ n $ in the definition of $\Phi:\R\times \S^n\rightarrow \S^n$ with $\abs{\beta}$.

Define the mapping $ \cL: \R\times \S^n\rightarrow \S^n $ as
\[
	\cL(\tau, H):=\Phi'((0,W); (\tau, H)),\quad (\tau,H)\in \R\times \S^n.
\]
Using \eqref{eq-dir-diff-new}, we see that the mapping $\cL(\cdot,\cdot)$ is F-differentiable at $(\tau, H)$ if and only if $\tau \neq 0$. Obviously, $\cL$ is locally Lipschitz continuous everywhere. Hence, $\partial_B \cL(0,0)$ is well-defined. Then, we have the following useful result that builds an insightful connection between $\partial_B \Phi$ and $\partial_B\cL$. To prove the result, we will basically follow the proof of \cite[Lemma 4]{chan2008constraint}. However, due to the second term on the right-hand side of  \eqref{eq-dir-diff-new}, we need to {modify} the proof {to make} the paper self-contained{; see Appendix \ref{proof-partial-Gamma} for more details.}

\begin{lemma}
	\label{lemma-partial-Gamma}
	Suppose that $W\in \S^n$ has the eigenvalue decomposition \eqref{eq-eig-W}. Then, it holds that
	\[
	\partial_B\Phi(0,W) = \partial_B\cL(0, 0).
	\]
\end{lemma}

Recall that $\Phi_{\abs{\beta}}$ is also locally Lipschitz continuous. Hence, $\partial_B \Phi_{\abs{\beta}}$ and $ \partial \Phi_{\abs{\beta}} $ are both well-defined. The following lemma provides an effective way to calculate $\partial_B\Phi(0,W)$ and $\partial \Phi(0, W)$.

\begin{lemma}
	\label{lemma-partial-B-Phi}
	For any $W\in \S^n$ with the spectral decomposition \eqref{eq-eig-W}. $V\in \partial_B\Phi(0,W)$ (respectively, $ \blue{\partial \Phi(0,W) }$) if and only if there exist $V_{\abs{\beta}}\in \partial_B\Phi_{\abs{\beta}}(0,0)$ (respectively, $ \partial \Phi_{\abs{\beta}}(0,0) $) and a scalar $v\in\{-1,1\}$ (respectively, $[-1,1]$) such that for all $(\tau, H)\in \R\times \S^n$,
	\[
	V(\tau, H) =
	\begin{pmatrix}
		\tilde H_{\alpha\alpha} & \tilde H_{\alpha\beta} & [\Omega_0(d)]_{\alpha \gamma} \circ \tilde H_{\alpha \gamma}  \\
		\tilde H_{\alpha \beta}^T & V_{\abs{\beta}}(\tau, \tilde H_{\beta\beta}) & 0 \\
		\tilde H_{\alpha \gamma}^T\circ [\Omega_0(d)]_{\alpha \gamma}^T & 0 & 0
	\end{pmatrix} P^T +  \frac{v\tau}{2}\sum_{i\in\alpha}p_ip_i^T,
	\]
	where $\tilde{H} := P^THP$. Moreover, $V_{\abs{\beta}}\in \partial_B\Phi_{\abs{\beta}}(0,0)$ if and only if there exist an orthogonal matrix $U\in \R^{\abs{\beta}\times\abs{\beta}}$ and a symmetric matrix
	\[
	\Omega_{\abs{\beta}}\in \left\{\Omega\in \S^{\abs{\beta}}: \Omega_{\abs{\beta}} = \displaystyle\lim_{k\rightarrow\infty} \Omega(\epsilon^k, z^k), (\epsilon^k, z^k)\rightarrow (0,0), \epsilon^k\neq 0, z_1^k\geq \dots \geq z_{\abs{\beta}}^k\right\}
	\]
	such that
	\[
	V_{\abs{\beta}}(0,Y) = U\left[\Omega_{\abs{\beta}}\circ(U^TYU)\right]U^T,\quad \forall Y\in \S^{\abs{\beta}}.
	\]
\end{lemma}
\begin{proof} See Appendix \ref{proof-partial-B-Phi}.
\end{proof}

As a corollary of Lemma \ref{lemma-partial-Pi} and Lemma \ref{lemma-partial-B-Phi}, we {can verify} that for any $V_0\in \partial_B\Pi_{\S_+^n}(W)$, there exists $V\in \partial_B\Phi(0, W)$ such that
\begin{equation}\label{eq-V0-VH}
	V_0(H) = V(0, H), \quad \forall H\in \S^n.
\end{equation}

We end this section by presenting a useful inequality for elements in  $ \partial \Phi(0, W) $.
\begin{lemma}	\label{lemma-V-PSD}
	Let $ W \in \S^n$ have the spectral decomposition \eqref{eq-eig-W}. Then for any $ V\in \partial\Phi(0,W) $,
	\[
		\inprod{H - V(0, H)}{V(0, H)} \geq 0, \quad \forall H\in \S^n.
	\]
\end{lemma}
\begin{proof} See Appendix \ref{proof-V-PSD}. 
\end{proof}

\section{A squared smoothing Newton method}
\label{smooth-newton}
In this section, we shall present our main algorithm, a squared smoothing Newton method {via} the Huber smoothing function. We then focus on analyzing its correctness, global convergence, and the fast local convergence rate under suitable regularity conditions. \blue{In the following, $\R,\; \R_+$, and $ \R_{++}$ denote the space of real numbers, the nonnegative orthant and the positive orthant, respectively.}

By the results presented in the previous section, we can define the smoothing function for $\cF$ based on the smoothing function $\Phi$ for $\Pi_{\S_+^n}$. In particular, let $\cE:\R\times \X\rightarrow {\R^m\times\S^n\times \S^n}$ be defined as
\begin{equation}
	\label{eq-G}
	\cE(\epsilon, X, y, Z) =
	\begin{pmatrix}
		\cA X {+ \kappa_p \abs{\epsilon}y  - b}  \\
		-\cA^* y - Z  + C \\
		(1+\kappa_c \abs{\epsilon})X - \Phi(\epsilon, X- Z)
	\end{pmatrix},
	\quad \forall (\epsilon, X, y, Z)\in \R\times \X,
\end{equation}
where $\kappa_p > 0,\, \kappa_c > 0$ are two given constants and $\X:= \S^n\times \R^m \times \S^n$. Then, $\cE$ is a continuously differentiable function around any $(\epsilon,X,y,Z)$ for any $\epsilon \neq  0$. Also, it satisfies
\[
\cE(\epsilon', X', y', Z') \rightarrow \cF(X, y, Z),\quad \textrm{as}\; (\epsilon', X', y', Z') \rightarrow (0, X, y, Z).
\]
Note that adding the perturbation terms $\kappa_p\abs{\epsilon}y$ and $\kappa_c\abs{\epsilon}X$ for constructing the smoothing function of $\cF$ is crucial in our algorithmic design, since it {ensures} that the proposed algorithm is well-defined (see Lemma \ref{lemma-nonsingular}).

Define the function $\hcE:\R\times \X \rightarrow \R\times {\R^m\times\S^n\times \S^n}$ by
\begin{equation}
	\label{eq-E}
	\hcE(\epsilon, X, y, Z) = \begin{pmatrix}
		\epsilon \\
		\cE(\epsilon, X, y, Z)
	\end{pmatrix},\quad \forall (\epsilon, X, y, Z) \in \R\times \X.
\end{equation}
Then solving the nonsmooth equation $\cF(X, y, Z) = 0$ is equivalent to solving the following system of nonlinear equations
\begin{equation}
	\label{eq-smooth-E}
	\hcE(\epsilon,X, y, Z) = 0.	
\end{equation}

Associated with the mapping $ \hcE $, we have the natural merit function  $\psi:\R\times \X\rightarrow \R_+$ that is defined as
\begin{equation}
	\label{eq-merit}
	\psi(\epsilon,X, y, Z) := \norm{\hcE(\epsilon, X, y, Z)}^2,\quad \forall  (\epsilon, X, y, Z) \in \R\times \X.
\end{equation}
Given $r \in {(0,1]}$, $\hat r\in (0,\infty)$ and $\tau\in (0,1)$, {we can} define two functions $\zeta:\R\times \X\rightarrow \R_+$ and $\eta: \R\times \X\rightarrow \R_+$ as follows:
\begin{equation}\label{eq-zeta}
	\zeta(\epsilon, X, y, Z) = r\min\left\{1, \norm{\hcE(\epsilon, X, y, Z)}^{1+\tau}\right\},\quad (\epsilon, X, y, Z) \in \R\times \X,
\end{equation}
and
\begin{equation}\label{eq-eta}
	\eta(\epsilon, X, y, Z) = \min\left\{1,\hat r\norm{\hcE(\epsilon, X, y, Z)}^{\tau}\right\},\quad (\epsilon, X, y, Z) \in \R\times \X.
\end{equation}
Then the inexact smoothing Newton can be described in Algorithm \ref{alg-ssnm}.

\begin{algorithm}[!]
    \centering
    \begin{algorithmic}[1]
        \State \textbf{Input:} $\hat \epsilon \in (0, \infty)$, $ r \in (0,1) $, $\hat r \in (0, \infty)$, ${\hat\eta} \in (0,1)$ be such that $\delta:= \sqrt{2}\max\{r\hat\epsilon, {\hat\eta}\} < 1$, $ \rho\in (0,1) $, $ \sigma\in (0,1/2) $, $ \tau\in (0,1] $, $\epsilon^0 = \hat \epsilon $, $ (X^0, y^0, Z^0)\in {\S^n \times} \R^m \times \S^n $.
        
        \For{$k\geq 0$}        
            \If{$ \hcE(\epsilon^k, X^k, y^k, Z^k) = 0 $}
                \State \textbf{Output:} $ (\epsilon^k, X^k, y^k, Z^k) $.
            \Else 
                \State Compute $\eta_k := \eta(\epsilon^k, X^k, y^k, Z^k)$ and $\zeta_k:=\zeta(\epsilon^k, X^k, y^k, Z^k)$.
                \State Solve the following equation
			\begin{equation}\label{eq-alg-direction}
				\hcE(\epsilon^k, X^k, y^k, Z^k) + \hcE'(\epsilon^k, X^k, y^k, Z^k)
				\begin{pmatrix}
					\Delta \epsilon^k \\ \Delta X^k \\ \Delta y^k \\ \Delta Z^k
				\end{pmatrix} =
				\begin{pmatrix}
					\zeta_k\hat \epsilon \\ 0 \\ 0 \\ 0
				\end{pmatrix}
			\end{equation}
			approximately such that
			\begin{equation}\label{eq-alg-inexactness}
                \norm{\cR_k} \leq \min\left\{\eta_k \norm{\cE(\epsilon^k, X^k, y^k, Z^k) + \cE_\epsilon'(\epsilon^k, X^k, y^k, Z^k)\Delta \epsilon^k}, \hat\eta \norm{\hcE(\epsilon^k, X^k, y^k, Z^k)}\right\},
			\end{equation}
			where $\Delta \epsilon^k := -\epsilon^k + \zeta_k\hat \epsilon$ and 
			$
				\cR_k := \cE(\epsilon^k, X^k, y^k, Z^k) + \cE'(\epsilon^k, X^k, y^k, Z^k)
    			\begin{pmatrix}
    				\Delta \epsilon^k \\ \Delta X^k \\ \Delta y^k \\ \Delta Z^k
    			\end{pmatrix}.
			$
            \State Compute $\ell_k$ as the smallest nonnegative integer $\ell$ satisfying
			\[
                \psi(\epsilon^k + \rho^\ell \Delta \epsilon^k, X^k + \rho^\ell \Delta X^k,  y^k + \rho^\ell \Delta y^k, Z^k + \rho^\ell \Delta Z^k) \leq [1 - 2\sigma (1 - \delta)\rho^\ell]\psi(\epsilon^k, X^k, y^k, Z^k).
			\]
            \State {Compute} $$(\epsilon^{k+1}, X^{k+1}, y^{k+1}, Z^{k+1}) = (\epsilon^k + \rho^{\ell_k} \Delta \epsilon^k, X^k + \rho^{\ell_k}\Delta X^k,  y^k + \rho^{\ell_k} \Delta y^k, Z^k + \rho^{\ell_k} \Delta Z^k).$$
            \EndIf
        \EndFor
    \end{algorithmic}
    \caption{A squared smoothing Newton method}
	\label{alg-ssnm}
\end{algorithm}

For the rest of this section, we shall show that Algorithm \ref{alg-ssnm} is well-defined and analyze its convergence properties. Note that the global convergence and the fast local convergence rate under the nonsingularity of $ \partial_B \hcE $ or $ \partial \hcE $ at solution points of nonlinear equations \eqref{eq-smooth-E} are studied extensively in the literature {(see e.g., \cite{chan2008constraint,gao2010calibrating})}. However, since we are using the Huber function, conditions that {ensure} the {nonsingularity} conditions need to be redeveloped.  {Furthermore, Algorithm \ref{alg-ssnm} allows one to specify the parameter $\tau\in (0,1]$ (which is used to control the rate of convergence) whereas existing results mainly focus on the case when $\tau = 1$.} {Consequently}, to make the paper self-contained, we also present the details of the analysis for our key results.

\subsection{Global convergence}
We first show that the linear system \eqref{eq-alg-direction} is well-defined and solvable for any $\epsilon^k > 0$. Hence, the inexactness {conditions} in \eqref{eq-alg-inexactness} {are} always achievable. This objective can be accomplished by showing that the coefficient matrix of the linear system \eqref{eq-alg-direction} is nonsingular for any $\epsilon^k > 0$. \blue{In the following exposition, $\cE'$ and $\hcE'$ denote the derivative of $\cE$ and $\hcE$, respectively, and $\cE'_\epsilon$ denotes the partial derivative of $\cE$ with respect to the first argument $\epsilon$.}

\begin{lemma} \label{lemma-nonsingular}
	For any $(\epsilon', X', y', Z') \in \R_{++}\times \X$, there exists an open neighborhood $\cU$ of $ (\epsilon', X', y', Z') $ such that $\hcE'(\epsilon, X, y, Z)$ is nonsingular for any $ (\epsilon, X, y, Z) \in \cU$  with $\epsilon \in \R_{++}$.
\end{lemma}
\begin{proof} 
	See Appendix \ref{proof-nonsingular}.
\end{proof}

The next task is to show that the line search procedure is well-defined, i.e., $\ell_k$ is finite for all $k\geq 0$. We will see that the inexactness condition $\norm{\cR_k}\leq {\hat\eta} \norm{\hcE(\epsilon^k, X^k, y^k, Z^k)}$ plays a fundamental role in our analysis. We also note that the inexactness condition
\[
	\norm{\cR_k}\leq \eta_k \norm{\cE(\epsilon^k, X^k, y^k, Z^k) + \cE_\epsilon'(\epsilon^k, X^k, y^k, Z^k)\Delta \epsilon^k}
\]
does not affect the correctness of Algorithm \ref{alg-ssnm} but will be crucial in analyzing the local convergence rate of Algorithm \ref{alg-ssnm}.

\begin{lemma}\label{lemma-step-size}
	For any $(\epsilon', X', y', Z') \in \R_{++}\times \X$, there exist an open neighborhood $\cU$ of $ (\epsilon', X', y', Z') $ and a positive scalar $\bar\alpha\in (0, 1]$ such that for any $ (\epsilon, X, y, Z) \in \cU$ with $\epsilon \in \R_{++}$ and $\alpha\in (0, \bar \alpha]$, it holds that
	\[
	\psi(\epsilon + \alpha \Delta \epsilon, X + \alpha\Delta X, y + \alpha \Delta y, Z + \alpha \Delta Z) \leq \left[1 - 2\sigma(1 - \delta)\alpha \right] \psi(\epsilon, X, y, Z),
	\]
	where {$\bDelta := (\Delta \epsilon; \Delta X; \Delta y; \Delta Z)\in \R \times \X$} satisfies
	\[
	\Delta \epsilon = -\epsilon + \zeta(\epsilon, X, y, Z)\hat \epsilon,\quad
	\norm{\cE(\epsilon, X, y, Z) + \cE'(\epsilon, X, y, Z) {\bDelta}}
	\leq {\hat\eta} \norm{\hcE(\epsilon, X, y, Z)}.
	\]
\end{lemma}
\begin{proof} See Appendix \ref{proof-step-size}. 
\end{proof}

As showed before, the nonsingularity of the coefficient matrix in \eqref{eq-alg-direction} requires $\epsilon^k$ \blue{being} positive. The next lemma shows that Algorithm \ref{alg-ssnm} generates $\epsilon^k$ that is lower bounded by $\zeta(\epsilon^k, X^k, y^k, Z^k)\hat\epsilon$. Thus, as long as the optimal solution is not found, $ \epsilon^k $ remains positive. To present the lemma, we need to define the following set
\[
	\cN := \left\{ (\epsilon, X, y, Z)\;:\; \epsilon \geq \zeta(\epsilon, X, y, Z)\hat\epsilon \right\}.
\]
\begin{lemma}
	\label{lemma-lb-epsilon}
	Suppose that for a given $k \geq 0$, $\epsilon^k \in \R_{++}$ and $(\epsilon^k, X^k, y^k, Z^k)\in \cN$. Then, for any $\alpha\in [0, 1]$ such that
	\[
		\psi(\epsilon^k + \alpha \Delta \epsilon^k, X^k + \alpha \Delta X^k, y^k + \alpha \Delta y^k, Z^k + \alpha \Delta Z^k) \leq [1 - 2\sigma (1 - \delta)\alpha]\psi(\epsilon^k, X^k, y^k, Z^k),
	\]
	it holds that
	\[
	(\epsilon^k + \alpha \Delta \epsilon^k, X^k + \alpha \Delta X^k,   y^k + \alpha \Delta y^k, Z^k + \alpha \Delta Z^k) \in \cN.
	\]
\end{lemma}
\begin{proof} See Appendix \ref{proof-lb-epsilon}.
\end{proof}

With those previous results, we can now establish the global convergence of Algorithm \ref{alg-ssnm}.

\begin{theorem}
	\label{thm-global-convergence}
	Algorithm \ref{alg-ssnm} is well-defined and generates an infinite sequence $\{(\epsilon^k, X^k, y^k, Z^k)\} \subseteq \cN$ with the property that any accumulation point $(\bar \epsilon, \bar X, \bar y, \bar Z)$ of $ \{(\epsilon^k, X^k, y^k, Z^k)\} $ (it exists if the solution set to the KKT system is nonempty and bounded)  is a solution of $\hcE(\epsilon, X, y, Z) = 0$, $ (\epsilon, X, y, Z)\in \R\times \X $.
\end{theorem}
\begin{proof} See Appendix \ref{proof-global-convergence}. 
\end{proof}

\subsection{Superlinear convergence rate}
We {next establish} the superlinear convergence rate of Algorithm \ref{alg-ssnm} under certain regularity conditions. The following lemma {is} useful for characterizing the nonsingularity of an element in $\partial\hcE$ at any accumulation point.

\begin{lemma}
	\label{lemma-delta-ineq}
	Suppose that $ (\bar X, \bar y, \bar Z)\in \X $ is a KKT-point. Let $ \bar X - \bar Z:= \bar W \in \S^n $ and $\bar V\in \partial\Phi(0,\bar W)$. Then, for any $\Delta X$ and $\Delta Z$ in $\S^n$ such that $\Delta X = \bar V(0, \Delta X - \Delta Z)$, it holds that $\inprod{\Delta X}{\Delta Z} \leq \Gamma_{\bar X}(-\bar Z, \Delta X)$, where $ \Gamma_{\bar X} $ is defined as in \eqref{eq-Gamma-B}.
\end{lemma}
\begin{proof} 
    See Appendix \ref{proof-delta-ineq}.
\end{proof}

Using the above lemma, we can establish the following equivalent relations.

\begin{proposition}
	\label{prop-nonsingular-cE}
	Let $ (\bar \epsilon, \bar X, \bar y, \bar Z) $ be such that $\hcE(\bar \epsilon, \bar X, \bar y, \bar Z) = 0$. Then the following statements are equivalent to each other.
	\begin{enumerate}
		\item The primal constraint nondegenerate condition holds at $\bar X$, and the dual constraint nondegenerate condition holds at $ (\bar y, \bar Z) $.
		\item Every element in $\partial_B\hcE(\bar \epsilon, \bar X, \bar y, \bar Z)$ is nonsingular.
		\item Every element in $\partial\hcE(\bar \epsilon, \bar X, \bar y, \bar Z)$ is nonsingular.
	\end{enumerate}
\end{proposition}
\begin{proof}
    See Appendix \ref{proof-nonsingular-cE}.	
\end{proof}

Typically, under certain {regularity} conditions, one is able to establish the local fast convergence rate of Newton-type methods. Indeed, we show in the next theorem that Algorithm \ref{alg-ssnm} admits a superlinear convergence rate under the primal and dual constraint nondegenerate conditions.

\begin{theorem}
	\label{thm-superlinear-convergence}
	Let $(\bar \epsilon, \bar X, \bar y, \bar Z) $ be an accumulation point of the infinite sequence $\{(\epsilon^k, X^k, y^k, Z^k)\} $ generated by Algorithm \ref{alg-ssnm}. Suppose that the primal constraint {nondegeneracy} condition holds at $\bar X$, and the dual constraint {nondegeneracy} condition holds at $ (\bar y, \bar Z) $. Then, the whole sequence $\{(\epsilon^k, X^k, y^k, Z^k)\} $ converges to $ (\bar \epsilon, \bar X, \bar y, \bar Z) $ superlinearly, i.e.,
	\[
		\norm{(\epsilon^{k+1} - \bar \epsilon, X^{k+1} - \bar X, y^{k+1} - \bar y, Z^{k+1} - \bar Z)} =  O\left(\norm{(\epsilon^{k} - \bar \epsilon, X^k - \bar X, y^{k} - \bar y, Z^{k} - \bar Z)}^{1+\tau}\right),
	\]
	where $\tau \in (0,1]$ controls how accurately one should solve the Newton system in \eqref{eq-alg-direction}.
\end{theorem}
\begin{proof}
    See Appendix \ref{proof-superlinear-convergence}. 
\end{proof}

\blue{
Analyzing the fast convergence rate of {smoothing Newton type methods for solving SDPs has been 
a continuing research direction} for decades. Early works (see, e.g., \cite{chen2003analysis, chen2003non, sun2004squared}) assumed the strict complementarity (i.e., $\bar X + \bar Z \in\mathbb{S}_{++}^n$) and primal-dual nondegeneracy conditions for establishing the local fast convergence rate. Later works \cite{kanzow2005quadratic,chan2008constraint}, attempted to drop the strict complementarity and establish the fast convergence rate merely under nondegeneracy conditions for smoothing Newton methods via the minimum smoothing function and the CHKS-smoothing function, respectively. Following the same research theme as \cite{kanzow2005quadratic,chan2008constraint}, our work establishes similar convergence properties for the Huber-based squared smoothing Newton method under the same nondegeneracy conditions without assuming the strict complementarity condition.
}

\begin{remark}[Convergence Properties of IPMs] \label{remark-conv-ipm}
	Here, we shall provide some comparisons on the singularity of the (generalized) Jacobian matrices arising from IPMs and those from the proposed smoothing Newton methods. Given an interior-point iterate $(X,y,Z)$ that is close to some KKT point $(X^*,y^*,Z^*)$ and denote $\mu:=\inprod{X}{Z}/n$. If $(X^*,Z^*)$ satisfies the strict complementarity condition (i.e., $X^*+Z^*\succ 0$) and the primal and dual nondegeneracy conditions hold at the KKT point, it is known that the Jacobian matrix of the KKT conditions has a bounded condition number even when $\mu > 0$ approaches 0 \cite{alizadeh1998primal}. Under the same conditions, the IPM with the so-called AHO direction is shown to have a superlinear or quadratic convergence rate \cite{alizadeh1998primal}. The key difference between our results and those for IPMs is that we do not need to assume the strict complementarity holds.
\end{remark}

\subsection{Numerical implementation} \label{subsection-implementaion}

{Given $\nu > 0$, one can check that the condition $X - \Pi_{\S_+^n}(X - Z) = 0$ is equivalent to the condition $ X - \Pi_{\S_+^n}(X - \nu Z) = 0 $. Thus in our implementation, we in fact solve the following nonlinear system (with a slight abuse of the notation $ \cE $):
\[
	\cE(\epsilon, X, y, Z) =
	\begin{pmatrix}
		\cA X {+ \kappa_p \abs{\epsilon}y  - b}  \\
		-\cA^* y - Z  + C \\
		(1+\kappa_c \abs{\epsilon})X - \Phi(\epsilon, X- \nu Z)
	\end{pmatrix},
	\quad \forall (\epsilon, X, y, Z)\in \R\times \X.
\]
Our numerical experience shows that introducing the parameter $\nu$ is important as it balances the norms of $X$ and $Z$, and thus improving the performance of the algorithm. However, the convergence properties established in the previous sections are not affected.
}

Note that one of the key {computational} tasks in Algorithm \ref{alg-ssnm} is to solve a system of linear equations for computing the search direction at each iteration. Here, we briefly explain how we solve the linear system of the following form:
\begin{equation}
	\label{eq-linsys-num}
	\begin{pmatrix}
		\cA & \mu_p I_m & 0 \\
		0 & -\cA^* & -\cI \\
		(1+\mu_c)\cI-V & 0 & {\nu} V
	\end{pmatrix}
	\begin{pmatrix}
		\Delta X \\ \Delta y \\ \Delta Z
	\end{pmatrix} =
	\begin{pmatrix}
		R_1 \\ R_2 \\ R_3
	\end{pmatrix},
\end{equation}
where $\mu_p := \kappa_p\abs{\epsilon} > 0$, $ \mu_c := \kappa_c\abs{\epsilon} > 0 $, $V {:=} \Phi'_2(\epsilon, X-{\nu}Z)$, $\cI$ is the identity mapping over $\S^n$ and $(R_1,R_2,R_3)\in \R^m\times \S^n \times \S^n$ is the right-hand-side vector constructed from the current iterate $(X,y,Z)\in \X$.  From the last two equations of \eqref{eq-linsys-num}, we have
\[
\Delta Z = -\cA^*\Delta y - R_2,\quad \Delta X = \left[(1+\mu_c)\cI - V\right]^{-1}(R_3 - {\nu} V\Delta Z),
\]
which implies that $\Delta X = \left[(1+\mu_c)\cI - V\right]^{-1}(R_3 + {\nu}V\cA^*\Delta y + {\nu}VR_2)$. Substituting this equality into the first equation in \eqref{eq-linsys-num}, we get
\[
\cA \left[(1+\mu_c)\cI - V\right]^{-1}(R_3 + {\nu}V\cA^*\Delta y + {\nu}VR_2) + \mu_p \Delta y = R_1,
\]
{which leads to the following smaller $ m\times m $} symmetric positive definite system:
\begin{equation}
	\left( \mu_p I_m + {\nu}\cA \left[(1+\mu_c)\cI - V\right]^{-1} V \cA^*\right) \Delta y =  R_1 - \cA \left[(1+\mu_c)\cI - V\right]^{-1}({\nu}VR_2 + R_3).
\label{eq-schur}
\end{equation}
We then apply the preconditioned conjugate gradient (PCG) method with a simple approximated diagonal preconditioner to solve the last linear system to get $\Delta y$. After obtaining $ \Delta y $, we can compute $\Delta X $ and $\Delta Z$ in terms of $\Delta y$. In our experiments, we set $ \kappa_p $ to be a small number, say $ \kappa_p = 10^{-10} $, while $ \kappa_c $ should be {larger} and may be changed for different classes of problems. Note that when $\epsilon$ is small, the coefficient matrix in \eqref{eq-schur} can also be highly ill-conditioned, which is a critical issue that also arises in IPMs and ALMs (when the primal nondegeneracy condition is not satisfied). To alleviate this issue, an effective preconditioner is needed. However, 
to focus on the smoothing Newton method itself, we refrain from diving into this direction and leave it for future research.

Next, we shall see how to evaluate the matrix-vector products involving $ \left[(1+\mu_c)\cI - V\right]^{-1} V $ which are needed {when solving \eqref{eq-schur}}. Suppose that $(\epsilon, X, y, Z)$ with $\epsilon > 0$ is given and that $W:= X-\nu Z$ has the spectral decomposition in \eqref{eq-eig-W}. Then, the linear mapping $V:\S^n\rightarrow\S^n$ takes the following form: $ V(H) = P\left[ \Omega(\epsilon, d) \circ \left(P^THP\right)\right] P^T$, $  \forall H\in \S^n  $, where $\Omega(\epsilon, d)\in \S^n$ is defined in \eqref{eq-first-diff} and $[\Omega(\epsilon, d)]_{ij}\in [0,1]$, for $1\leq i, j\leq n$.
{Consider the following} three index sets: $\alpha:= \left\{ i\;:\; d_i\geq \epsilon\right\}$, $\beta:= \left\{ i\;:\; 0 < d_i < \epsilon\right\}$, and $\gamma:= \left\{ i\;:\; d_i\leq 0\right\}$.
Then, we can simply write $ \Omega(\epsilon, d) $ as
\[
	\Omega(\epsilon, d) = \begin{pmatrix}
		E_{\alpha\alpha} & \Omega_{\alpha\beta} & \Omega_{\alpha\gamma} \\
		 \Omega_{\alpha\beta}^T & \Omega_{\beta\beta} & \blue{\Omega_{\beta\gamma}} \\
		 \Omega_{\alpha\gamma}^T & \blue{\Omega_{\beta\gamma}^T} & 0
	\end{pmatrix},
\]
where $E_{\alpha\alpha}\in \R^{\abs{\alpha}\times \abs{\alpha}}$ is the matrix of all ones, and $\Omega_{\alpha\beta}\in \R^{\abs{\alpha}\times \abs{\beta}}$, $\Omega_{\beta\beta}\in \R^{\abs{\beta}\times \abs{\beta}}$, $\Omega_{\alpha\gamma}\in \R^{\abs{\alpha}\times \abs{\gamma}}$ and $\blue{\Omega_{\beta\gamma}\in \R^{\abs{\beta}\times \abs{\gamma}}}$ with all their entries belonging to the interval {$ (0,1) $.} We also partition the orthogonal matrix $P$ as $P = \begin{pmatrix}
	P_\alpha & P_\beta & P_\gamma
\end{pmatrix}$ accordingly. Define the matrix $ \hat\Omega\in \S^n$ as $ [\hat \Omega]_{ij}:= [\Omega(\epsilon,d)]_{ij}/(1 + \mu_c - [\Omega(\epsilon,d)]_{ij}) $ for $i, j = 1,\dots, n$, which takes the following form
\[
	\hat \Omega = \begin{pmatrix}
		\frac{1}{\mu_c} E_{\alpha\alpha} & \hat \Omega_{\alpha\beta} & \hat \Omega_{\alpha\gamma} \\
		\hat \Omega_{\alpha\beta}^T & \hat \Omega_{\beta\beta} & \blue{\hat \Omega_{\beta\gamma}} \\
		\hat \Omega_{\alpha\gamma}^T & \blue{\hat\Omega_{\beta\gamma}^T} & 0
	\end{pmatrix},
\]
Then one can check that $ \left[(1+\mu_c)\cI - V\right]^{-1} V (H) = P\left[ \hat \Omega \circ \left(P^THP\right)\right] P^T $, $\forall H\in \S^n$. If $\abs{\alpha} + \abs{\beta} \ll n$, one may use the following scheme to compute
\[
	\begin{aligned}
		\left[(1+\mu_c)\cI - V\right]^{-1} V (H) =  &\; \frac{1}{\mu_c}P_\alpha\big(P_\alpha^THP_\alpha\big)P_\alpha^T + P_\beta\big(\hat\Omega_{\beta\beta}\circ (P_\beta^T HP_\beta)\big)P_\beta^T \\
		&\; + P_\alpha\big(\hat \Omega_{\alpha\beta}\circ (P_\alpha^THP_\beta)\big)P_\beta^T
+  {\Big(P_\alpha\big(\hat \Omega_{\alpha\beta}\circ (P_\alpha^THP_\beta)\big)P_\beta^T \Big)^T}
 \\
		&\; + P_\alpha\big(\hat \Omega_{\alpha\gamma}\circ (P_\alpha^THP_\gamma)\big)P_\gamma^T
+ \Big(P_\alpha\big(\hat \Omega_{\alpha\gamma}\circ (P_\alpha^THP_\gamma)\big)P_\gamma^T \Big)^T \\
&\; + \blue{P_\beta\big(\hat \Omega_{\beta\gamma}\circ (P_\beta^THP_\gamma)\big)P_\gamma^T
+ \Big(P_\beta\big(\hat \Omega_{\beta\gamma}\circ (P_\beta^THP_\gamma)\big)P_\gamma^T \Big)^T}.
	\end{aligned}
\]
On the other hand, if $n - \abs{\alpha} \ll n$, one may consider using the following scheme
to compute
\[
	\begin{aligned}
		\left[(1+\mu_c)\cI - V\right]^{-1} V (H) = &\; \frac{1}{\mu_c}H -  P\left[ \tilde \Omega \circ \left(P^THP\right)\right] P^T ,
	\end{aligned}
\]
where $\tilde \Omega := \frac{1}{\mu_c}E - \hat \Omega$ would have more zero entries than $\hat\Omega$. As a consequence, since the Huber function maps any non-positive number to zero, we can exploit the sparsity structure in $\hat \Omega$ or {$\tilde \Omega$} to cut down the computational cost. However, if the CHKS-smoothing function is used, we will get a dense counterpart and the aforementioned sparsity structure will be lost. Theoretically, both Huber-based and CHKS-based smoothing Newton methods share the same convergence properties. From our numerical experience, the practical performance of both methods in terms of the total number of CG iterations are also similar. Thus, reducing the computational costs in matrix-vector multiplications definitely makes the Huber-based method more efficient than the CHKS-based method. This also explains why we choose the Huber function instead of the CHKS function to perform the smoothing of the projection operator. 
Other than the CHKS smoothing function, one can also compare the practical performance of the Huber function with other smoothing functions, but we leave it as a topic for future investigation.

When compared with IPMs, the proposed algorithm demonstrates some advantages from the computational perspective. To see this, let $\mu = \inprod{X}{Z}/n$ be defined as in Remark \ref{remark-conv-ipm}. It is known that an IPM typically computes the search direction in each iteration via solving an $m\times m$ linear system whose coefficient matrix is defined as:
\[
\mathcal{M}_{\rm IPM}h:= \mathcal{A}P(D\circ (P^T(\mathcal{A}^*h)P))P^T,\quad \forall h\in \mathbb{R}^m,
\]
where $P\in \mathbb{R}^{n\times n}$ and $D\in \mathbb{R}_{++}^{n\times n}$ are fully dense matrices depending on the iterate $(X,Z)$ and hence also $\mu$ \cite{toh2002solving}. In general, the matrix $\mathcal{M}_{\rm IPM}$ is asymptotically singular when $\mu\to 0$, i.e., the condition number of $\mathcal{M}_{\rm IPM}$ will grow to infinity as $\mu\to 0$, even under strict complementarity, and primal and dual nondegeneracy conditions \cite{alizadeh1998primal}. Moreover, since $P$ and $D$ are fully dense, it is clear that computing the matrix-vector product involving $\mathcal{M}_{\rm IPM}$ would cost at least $8n^3$ arithmetic operations. Similar to IPMs, the $m \times m$ linear system \eqref{eq-schur} that is used to compute the search direction in the smoothing Newton method can also be asymptotically singular when $\epsilon\to 0$. The key difference is that the perturbation terms $\kappa_p|\epsilon|y$ and $\kappa_c|\epsilon|X$ introduced when constructing the smoothing function of $\mathcal{F}$ help to improve the conditioning of the underlying coefficient matrix. Moreover, from the view point of computing the matrix-vector product, we emphasize that the coefficient matrix in \eqref{eq-schur} can exploit the low-rank (or high-rank) property of the matrix $X-\nu Z$ as mentioned above. For example, if $r:=|\alpha|+|\beta| \ll n$, then the matrix-vector product only costs $O(n^2r)$ arithmetic operations, which could be much cheaper than that of an IPM iteration.

We observe from our numerical tests that when the dual iterates $(y^k, Z^k)$ does not make {a} significant progress, it is helpful for us to project the primal iterate $X^k$ onto the affine subspace ${\cH_k}:=\left\{X\in S^n: \cA(X) = b, \inprod{C}{X} = \inprod{b}{{y^k}}\right\}$. To perform such a projection operation, we only need to perform a Cholesky factorization for a certain operator (depending only on $\cA$ and $C$) once at the beginning of the Algorithm \ref{alg-ssnm}. However, in the case when the factorization fails (i.e., the operator is no positive definite) or $m$ is too large, we will not perform such projections.


\section{Numerical experiments}
\label{numerical-exp}
To evaluate the practical performance of Algorithm \ref{alg-ssnm} described in the last section, we conduct numerical experiments {to} solve various classes of SDPs that are commonly {tested} in the literature. 

\subsection{Experimental settings and implementation}
Similar to \cite{yang2015sdpnal}, the following relative KKT residues are used as the termination criteria of our algorithm:
\[
\eta_p := \frac{\norm{\cA X - b}}{1+\norm{b}},\quad \eta_d := \frac{\norm{\cA^*y + Z - C}}{1+\norm{C}},\quad \eta_c := \frac{\norm{X - \Pi_{\S_+^n}(X-Z)}}{1+\norm{X} + \norm{Z}}.
\]
Particularly, we terminate the algorithm as long as $\eta_{KKT}:= \max\{\eta_p,\eta_d,\eta_c\} \leq {\tt tol}$ where $ {\tt tol} > 0$ is a given tolerance. Moreover, denote the {maximum} number of iterations of Algorithm \ref{alg-ssnm} as {\tt maxiter}. We also stop the algorithm when the iteration count reaches this number. In our experiments, we set $ {\tt tol} = 10^{-6} $ and $ {\tt maxiter} = 50 $. 

For more efficiency, we {apply} a certain first-order method to generate a starting point $ (X^0, y^0, Z^0) $ {to warmstart Algorithm \ref{alg-ssnm}}. Our choice is the routine based on a semi-proximal ADMM method \cite{yang2015sdpnal}. The stopping tolerance (in terms of the maximal relative KKT residual, i.e., $\eta_{KKT}$) and the maximum number of iterations for the warmstarting phase are denoted by $ {\tt tol}_0 $ and $ {\tt maxiter}_0 $, respectively. In our experiments, we set \blue{$ {\tt maxiter}_0 = 1000 $}, and the computational time spent in  the warmstarting phase will be included in the total computational time. We also notice that the performance of Algorithm \ref{alg-ssnm} depends sensitively on the choice of $ {\tt tol}_0 $. Hence, we set the value of $ {\tt tol}_0 $ differently for different classes of SDPs for more efficiency.

For the parameters required in Algorithm \ref{alg-ssnm}, we set $ r = \hat r = 0.6 $, $ \eta = \tau = 0.2 $, $ \rho = 0.5 $ and $ \sigma = 10^{-8} $. However, since the initial smoothing parameter $ \hat \epsilon > 0 $ affects the performance of Algorithm \ref{alg-ssnm}, it is chosen differently for different classes of problems.

\blue{
The baseline solvers to be compared with are \texttt{SDPLR} \cite{burer2003nonlinear}, \texttt{ManiSDP} \cite{wang2023solving}, and {\tt SDPNAL+} \cite{yang2015sdpnal} \footnote{Based on our numerical experience, interior point methods such as Mosek become inefficient and/or incapable to handle SDPs with large $m$, say, $m\geq 10,000$, and first-order methods are typically quite slow to reach our targeted accuracy. Moreover, to the best of our knowledge, there is no publicly available implementation for the CHKS-based smoothing Newton method.}. We use their default settings but only set the stopping tolerance to be $10^{-6}$. 
}

Finally, we should mention that our algorithm is implemented in MATLAB \blue{(R2023b) and all the numerical experiments are conducted on a Windows PC having 13th Gen Intel cores (i7-13700K) and 64GB of RAMs.}

\subsection{Testing examples}
\label{examples}
\begin{example}[MaxCut-SDP]
	\label{eg-maxcut}
	Consider the SDP relaxation \cite{goemans1995improved} of the maximum cut problem {of a graph} which takes the form of
	\begin{equation*}
		\min_{X\in \S^n}\; \inprod{C}{X} \quad \st \quad \mathrm{diag}(X) = {e}, \; X\in \S_+^n,
	\end{equation*}
	where ${e}\in \R^n$ denotes the vector of all ones and $C:= -(\mathrm{diag}(We_n) - W)/4$ {with $W$ being the weighted adjacency matrix of the underlying graph}. The above SDP problem has been a commonly used test problem for evaluating the performance of different SDP solvers. {A popular} data set for this problem is the GSET collection of randomly generated graphs. The GSET is available at: \url{https://web.stanford.edu/~yyye/yyye/Gset/}.
\end{example}

\begin{example}[Theta-SDP]
	\label{eg-theta}
	Let $G = (V, E)$ be a graph with $ n $ nodes $V$ and edges $E$. A stable set of $G$ is a subset of $V$ containing {no adjacent nodes}. The stability number $\alpha(G)$ is the cardinality of a maximal stable set of $G$. More precisely, it holds that
	\[
	\alpha(G) := \max_{x\in \R^n} \;\left\{ e^Tx : x_ix_j = 0, (i,j) \in E, x\in \{0,1\}^n\right\}.
	\]
	{However}, computing $ \alpha(G) $ is difficult. A notable lower bound of $ \alpha(G) $ is called the Lov\'{a}sz theta number \cite{lovasz1979shannon} which is defined as
	\[
	\theta(G) := \max_{X\in \S^n} \; \inprod{ee^T}{X} \quad \st \quad \inprod{E_{ij}}{X} = 0, (i,j)\in E,\; \inprod{I}{X} = 1,\; X\in \S_+^n,
	\]
	where $E_{ij} = e_ie_j^T + e_je_i^T\in \S^n$. {For our experiments, the test} data sets are chosen from \cite[Section 6.3]{zhao2010newton}.
\end{example}

\begin{example}[BIQ-SDP]
	\label{eg-biq}
	The NP-hard binary integer quadratic programming  (BIQ) problem has the following form:
	\[
	\min_{x\in \R^n} \; \frac{1}{2}\inprod{x}{Qx} + \inprod{c}{x}\quad \st\quad x\in \{0,1\}^n.
	\]
	It has many important practical applications due to its modeling power in representing the structure of graphs; see for example \cite{kochenberger2014unconstrained} for a recent survey. Under some mild conditions, Burer showed in \cite{burer2009copositive} the BIQ problem can be reformulated as a completely positive conic programming problem. Since the completely positive cone is numerically intractable, we consider its SDP relaxation:
	\[
	\min_{X\in \S^n, x\in \R^n, \alpha \in \R}\; \frac{1}{2}\inprod{Q}{X} + \inprod{c}{x} \quad \st\quad \mathrm{diag}(X) = x,\; \alpha = 1, \begin{pmatrix}
		X & x^T \\ x & \alpha
	\end{pmatrix} \in \S_+^{n+1}.
	\]
	In our experiments, the matrix $Q$ and the vector $c$ is obtained from the BIQMAC library \cite{wiegele2007biq}.
\end{example}

\begin{example}[Clustering-SDP]
	\label{eg-rcp}
	The clustering problem {is to group} a set of data points into several clusters. The problem is in general NP-hard. According to \cite{peng2007approximating}, we can consider the following SDP relaxation of the clustering problem:
	\[
	\min_{X\in \S^n}\; -\inprod{W}{X}\quad \st\quad Xe_n = e_n, \; \inprod{I}{X} = k,\; X\in \S_+^n,
	\]
	where $W$ is the affinity matrix whose entries represent the {pairwise} similarities of the objects in the input data set. In our experiments, {the test} data sets are obtained {from} the UCI Machine Learning Repository: \url{http://archive.ics.uci.edu/ml/datasets.html}. For more information on how we generate {the test} SDP problems from the raw input data sets, readers are \blue{referred} to \cite{yang2015sdpnal}.
\end{example}

\begin{example}[Tensor-SDP]
	\label{eg-tensor}
	Consider the {following} SDP relaxation {of a} rank-1 tensor {approximation problem} \cite{nie2014semidefinite}:
	\[
		\max_{y\in \R^{\N_m^n}} \; \inprod{f}{y} \quad \st\quad M(y)\in \S_+^n, \;\inprod{g}{y} = 1,
	\]
	where $\N_m^n:=\left\{t = (t_1,\dots, t_n)\in \N^n: t_1+\dots+t_n = m\right\}$ and
	{$ M(y) $ is a linear pencil in $ y $.} The dual of the above problem is given by
	\[
		\min_{\gamma\in \R, X\in \S^n}\; \gamma \quad \st \quad \gamma g - f = M^*(X), \;X\in \S^n_+,
	\]
	{where $ M^* $ is the adjoint of $ M $. The above problem} can be transformed into a standard SDP; see \cite{yang2015sdpnal} for more details.
\end{example}

\subsection{Computational results}\label{comp-results}
In this subsection, we present our numerical results for each {class of SDP problems} in detail. We use tables to summarize our results. We report the sizes (i.e., $m$ and $n$) of the tested problems, number of iterations, total computational times (i.e., {\tt cpu}) in seconds, relative KKT residues (i.e., $ \eta_p $, $ \eta_d $ and $ \eta_c $), and objective function values (i.e., $ \inprod{C}{X} $). \blue{In particular, under the column labeled ''\texttt{Problem}'', we first present the name of the tested problem, followed by the matrix dimension $n$ and the number of linear constraints $m$, respectively. For the number of iterations associated with \texttt{SDPNAL+} and Algorithm \ref{alg-ssnm}, the first entry represents the number of iterations for the warmstarting phase, the second  is the number of main iterations, and the last  is the total number of PCG iterations needed for solving the linear systems when computing the search directions. For \texttt{SDPLR}, the first number stands for the total number of the ALM iterations and the second number is the total number of iterations of the limited-memory BFGS method that is applied to solve the (nonconvex) ALM subproblems. Last, for \texttt{ManiSDP}, the first item represents the number of iterations for the outer loop, the second  is the iteration counts for the Riemannian trust-region method used in \texttt{Manopt} \cite{manopt}, and the third number is the total number of matrix-vector multiplications. Moreover, KKT residues that are larger than $10^{-5}$ are highlighted in bold text.} As mentioned before, since the performance of our algorithm may depend sensitively on the choices of parameters $ {\tt tol}_0 $, $ \hat \epsilon $ and $ \kappa_c $, we also report the values for these parameters in the captions of {the} presented tables. The computational results for Example \ref{eg-maxcut}--\ref{eg-tensor} are presented in Table \ref{table-maxcut}--\ref{table-tensor}, respectively.

From the results in Tables \ref{table-maxcut}--\ref{table-rcp}, we observe that our proposed algorithm performs better than {\tt SDPNAL+}. In fact, Algorithm \ref{alg-ssnm} is usually several times more efficient than {\tt SDPNAL+}. These results have indeed shown that Algorithm \ref{alg-ssnm} is efficient. In \blue{terms} of accuracy, we see that both algorithms are able to compute optimal solutions with $\eta < 10^{-6}$ for almost all the tested problems, which further shows that both algorithms are robust and suitable for solving SDPs relatively accurately. We observe that for the SDPs in Tables \ref{table-maxcut}--\ref{table-rcp}, their optimal solutions are usually of very low-rank. Consequently, the primal constraint nondegenerate condition usually fails to hold, and {\tt SDPNAL+} requires more computational effort to converge than Algorithm \ref{alg-ssnm}. \blue{For those SDPs with low-rank solutions, one expects \texttt{SDPLR} and \texttt{ManiSDP} to have excellent performance. Indeed, we observe that \texttt{SDPLR} and \texttt{ManiSDP} share comparable performance and outperform \texttt{SDPNAL+} and Algorithm \ref{alg-ssnm} on Example \ref{eg-maxcut} {consisting of max-cut SDPs}. We note that the feasible set of the max-cut SDP defines	a smooth manifold, hence, \texttt{ManiSDP} essentially applies \texttt{Manopt} for solving an optimization over this manifold.
For Example \ref{eg-theta}, we see that \texttt{ManiSDP} outperforms \texttt{SDPLR} but 
is usually becomes less efficient than \texttt{SDPNAL+} and Algorithm \ref{alg-ssnm}. Also, we observe that \texttt{SDPLR} and \texttt{ManiSDP} typically output less accurate solutions in the sense that the KKT residues, especially $\eta_c$, are much worse than those of \texttt{SDPNAL+} and Algorithm \ref{alg-ssnm}. 
On Example \ref{eg-biq},
\texttt{SDPLR} performs better than \texttt{SDPNAL+} and \texttt{ManiSDP} but worse than Algorithm \ref{alg-ssnm}. Moreover, \texttt{ManiSDP} fails to solve all the tested problems for Example \ref{eg-biq}. This could be due to the fact that there is no explicit manifold structure for \texttt{ManiSDP} to take advantage of \footnote{\blue{In this case, we use the routine designed for solving general SDPs with arbitrary linear constraints provided by  \url{https://github.com/wangjie212/ManiSDP-matlab}.}}. 
{On the other hand,} for problems arising from Example \ref{eg-rcp}, there is a fixed-trace constraint which leads to an explicit manifold structure. In this case, \texttt{ManiSDP} shows the best performance for problems with $n \leq 1000$. However, \texttt{ManiSDP} fails to  solve the problem \texttt{spambase-large.2} with $n = 1500$. We also observe that \texttt{SDPLR} fails to solve half of the tested problems. As a conclusion, \texttt{SDPNAL+} and Algorithm \ref{alg-ssnm} show excellent performance in terms of efficiency and robustness, while the efficiency and robustness of \texttt{SDPLR} and \texttt{ManiSDP} can be sensitive to the problem being solved.}

From Table \ref{table-tensor}, we see that Algorithm \ref{alg-ssnm} is able to solve all the problems but perform much worse than {\tt SDPNAL+} in terms of computational time. The reason is that for those SDPs, the ranks of {the} optimal solutions are high. In fact, we always observe that the rank is nearly $ n-1 $, where $ n $ denotes the matrix dimension. In such cases, the primal constraint nondegenerate condition usually holds, making {\tt SDPNAL+} to have a very fast convergence rate. On the other hand, for these SDPs, because the dual constraint nondegeneracy condition typically fails to hold, Algorithm \ref{alg-ssnm} would require more computational effort than {\tt SDPNAL+} to solve the problems. \blue{Surprisingly, though \texttt{SDPLR} is designed for solving SDPs with low-rank solutions, it sometimes shows promising performance when solving SDPs with high-rank solutions, as indicated by Table \ref{table-tensor}. {In comparison}, \texttt{ManiSDP} fails to solve most of the tested problems due to the fact that the solutions are of high rank and there is no explicit manifold structure to utilize.}

Based on the relative performance of {\tt SDPNAL+} and Algorithm \ref{alg-ssnm} in Tables \ref{table-maxcut}--\ref{table-tensor}, we may conclude that when the primal constraint nondegeneracy condition is likely to hold at the optimal solution, {\tt SDPNAL+} would be preferable. However, when such prior information is not available, Algorithm \ref{alg-ssnm} could be a good choice based on the promising numerical performance demonstrated in Tables \ref{table-maxcut}--\ref{table-rcp}, even though the primal constraint nondegeneracy condition may not hold, \blue{based on our empirical observation}. Note finally that for our algorithm, the errors $\eta_p$ and $\eta_d$ are usually very small while $\eta_c$ is driven to a value slightly smaller than ${\tt tol} = 10^{-6}$. On the contrary, {\tt SDPNAL+} keeps $\eta_c$ to be very small while progressively decreases $\eta_p$ and $\eta_d$.

\small  
\blue{
\begin{longtable}{|l|l|r|r|c|c|}
	\caption{Computational results for Example \ref{eg-maxcut} with ${\tt tol}_0 = 10^{-3} $, $ \hat \epsilon = 10^{-3} $, {$ \nu = 5\times 10^{3}$} and $ {\kappa_c = 10^2}$.}
	\label{table-maxcut} \\
	\hline
	\texttt{Problem} & \texttt{Solver}    & \texttt{Iteration} & \texttt{Time} & \texttt{KKT} & \texttt{Objective} \\ \hline
	\endfirsthead
	\multicolumn{6}{c}%
	{{\bfseries Table \thetable\ continued from previous page}} \\
	\hline
	\texttt{Problem} & \texttt{Solver}    & \texttt{Iteration} & \texttt{Time} & \texttt{KKT} & \texttt{Objective} \\ \hline
	\endhead

	\multirow{2}{*}{\texttt{g1}} & \texttt{SDPLR} & 20, 356 &  2.4 & 2.2e-08, 0.0e+00, 1.5e-06 & -1.2083198e+04 \\
	& \texttt{ManiSDP} & 3, 74, 991 &  0.9 & 2.3e-16, 0.0e+00, 1.1e-12 & -1.2083198e+04 \\
	800 & \texttt{SDPNAL+} & 658, 12, 121 & 75.4 & 5.4e-14, 1.2e-06, 1.4e-06 & -1.2083198e+04 \\
	800 & \texttt{Ours} & 73, 6, 425 & 11.4 & 0.0e+00, 7.2e-13, 2.1e-07 & -1.2083197e+04 \\ \hline
	\multirow{2}{*}{\texttt{g2}} & \texttt{SDPLR} & 20, 340 &  2.3 & 5.8e-08, 0.0e+00, 7.3e-06 & -1.2089430e+04 \\
	& \texttt{ManiSDP} & 3, 80, 1230 &  0.9 & 2.5e-16, 0.0e+00, 4.6e-12 & -1.2089430e+04 \\
	800 & \texttt{SDPNAL+} & 734, 11, 125 & 94.9 & 2.1e-14, 2.9e-06, 6.0e-07 & -1.2089431e+04 \\
	800 & \texttt{Ours} & 73, 6, 430 & 23.4 & 0.0e+00, 6.9e-13, 6.2e-07 & -1.2089427e+04 \\
\hline 
\end{longtable}

\begin{longtable}{|l|l|r|r|c|c|}
	\caption{Computational results for Example \ref{eg-theta} with ${\tt tol}_0 = 10^{-4} $, $ \hat \epsilon = 10^{-3} $ , {$ \nu = 5\times 10^{3}$} and $ \kappa_c = 10^{3}$.}
	\label{table-theta} \\
	\hline
	\texttt{Problem} & \texttt{Solver}    & \texttt{Iteration} & \texttt{Time} & \texttt{KKT} & \texttt{Objective} \\ \hline
	\endfirsthead
	\multicolumn{6}{c}%
	{{\bfseries Table \thetable\ continued from previous page}} \\
	\hline
	\texttt{Problem} & \texttt{Solver}    & \texttt{Iteration} & \texttt{Time} & \texttt{KKT} & \texttt{Objective} \\ \hline
	\endhead

	\multirow{2}{*}{$\texttt{theta4}$} & \texttt{SDPLR} & 17, 34750 & 20.3 & 1.0e-06, 0.0e+00, 2.0e-07 & -5.0321375e+01 \\
	& \texttt{ManiSDP} & 50, 200, 2871 &  1.9 & 1.6e-07, 9.0e-18, 1.8e-06 & -5.0321226e+01 \\
	200 & \texttt{SDPNAL+} & 451, 9, 352 &  3.5 & 3.9e-08, 1.6e-07, 2.7e-16 & -5.0321223e+01 \\
	1949 & \texttt{Ours} & 198, 2, 75 &  1.4 & 1.8e-19, 5.6e-08, 1.1e-07 & -5.0321190e+01 \\ \hline
	\multirow{2}{*}{$\texttt{theta42}$} & \texttt{SDPLR} & 18, 18742 & 14.5 & 1.0e-06, 0.0e+00, \textbf{1.2e-04} & -2.3931747e+01 \\
	& \texttt{ManiSDP} & 110, 440, 6975 &  5.8 & 2.0e-07, 1.3e-17, \textbf{2.5e-05} & -2.3931705e+01 \\
	200 & \texttt{SDPNAL+} & 176, 4, 496 &  1.6 & 1.2e-07, 9.1e-08, 4.6e-16 & -2.3931708e+01 \\
	5986 & \texttt{Ours} & 79, 2, 99 &  0.6 & 2.1e-16, 7.8e-08, 1.3e-07 & -2.3931708e+01 \\
\hline
	
\end{longtable}

\begin{longtable}{|l|l|r|r|c|c|}
	\caption{Computational results for Example \ref{eg-biq} with ${\tt tol}_0 = 5\times 10^{-1} $, $ \hat \epsilon = 10^{-3} $, {$ \nu = 5\times 10^{3}$} and $ \kappa_c = 10^{2}$.}
	\label{table-biq} \\
	\hline
	\texttt{Problem} & \texttt{Solver}    & \texttt{Iteration} & \texttt{Time} & \texttt{KKT} & \texttt{Objective} \\ \hline
	\endfirsthead
	\multicolumn{6}{c}%
	{{\bfseries Table \thetable\ continued from previous page}} \\
	\hline
	\texttt{Problem} & \texttt{Solver}    & \texttt{Iteration} & \texttt{Time} & \texttt{KKT} & \texttt{Objective} \\ \hline
	\endhead

	\multirow{2}{*}{$\texttt{be150.3.1}$} & \texttt{SDPLR} & 20, 5189 &  4.4 & 8.3e-07, 0.0e+00, 1.2e-08 & -2.0175194e+04 \\
	& \texttt{ManiSDP} & 290, 1447, 16714 &  5.0 & \textbf{1.2e+01}, 3.0e-17, \textbf{2.9e-05} & -6.8971132e+04 \\
	151 & \texttt{SDPNAL+} & 655, 26, 5808 &  7.7 & 1.3e-11, 4.8e-06, 9.8e-10 & -2.0175328e+04 \\
	151 & \texttt{Ours} & 140, 11, 536 &  1.0 & 3.3e-16, 5.3e-12, 5.9e-07 & -2.0174310e+04 \\ \hline
	\multirow{2}{*}{$\texttt{be150.8.1}$} & \texttt{SDPLR} & 21, 10903 &  1.3 & 5.8e-07, 0.0e+00, 4.4e-08 & -2.9671661e+04 \\
	& \texttt{ManiSDP} & 550, 2750, 35913 & 10.1 & \textbf{2.5e+00}, 2.9e-17, 9.5e-06 & -5.5643145e+04 \\
	151 & \texttt{SDPNAL+} & 800, 25, 6219 &  9.4 & 3.8e-07, 7.6e-07, 2.2e-15 & -2.9671637e+04 \\
	151 & \texttt{Ours} & 127, 9, 341 &  0.8 & 3.0e-16, 8.9e-12, 9.2e-08 & -2.9671499e+04 \\
\hline
	
\end{longtable}

\begin{longtable}{|l|l|r|r|c|c|}
	\caption{Computational results for Example \ref{eg-rcp} with ${\tt tol}_0 =  10^{-1} $, $ \hat \epsilon = 10^{-1} $, {$ \nu = 5\times 10^{3}$} and $ \kappa_c = 5 \times 10^{3}$.}
	\label{table-rcp} \\
	\hline
	\texttt{Problem} & \texttt{Solver}    & \texttt{Iteration} & \texttt{Time} & \texttt{KKT} & \texttt{Objective} \\ \hline
	\endfirsthead
	\multicolumn{6}{c}%
	{{\bfseries Table \thetable\ continued from previous page}} \\
	\hline
	\texttt{Problem} & \texttt{Solver}    & \texttt{Iteration} & \texttt{Time} & \texttt{KKT} & \texttt{Objective} \\ \hline
	\endhead
	%
	
	\multirow{2}{*}{$\texttt{soybean-large.2}$} & \texttt{SDPLR} & 19, 184 &  0.2 & 1.3e-07, 0.0e+00, 5.9e-07 & -1.0581914e+04 \\
	& \texttt{ManiSDP} & 11, 44, 86 &  0.3 & 1.3e-08, 2.3e-17, 4.4e-15 & -1.0581910e+04 \\
	307 & \texttt{SDPNAL+} & 300, 5, 23 &  4.6 & 1.1e-07, 4.7e-09, 7.5e-17 & -1.0581910e+04 \\
	308 & \texttt{Ours} & 10, 10, 48 &  0.8 & 5.5e-16, 1.8e-13, 3.1e-07 & -1.0581898e+04 \\ \hline
	\multirow{2}{*}{$\texttt{spambase-small.2}$} & \texttt{SDPLR} &  & & &  \\
	& \texttt{ManiSDP} & 23, 92, 143 &  0.5 & 2.9e-07, 9.5e-16, 3.4e-13 & -1.1477648e+08 \\
	300 & \texttt{SDPNAL+} & 300, 5, 12 &  4.7 & 4.8e-09, 1.7e-07, 3.4e-17 & -1.1477648e+08 \\
	301 & \texttt{Ours} & 10, 8, 37 &  0.6 & 6.2e-16, 1.4e-11, 1.6e-07 & -1.1477619e+08 \\
\hline
\end{longtable}

\begin{longtable}{|l|l|r|r|c|c|}
	\caption{Computational results for Example \ref{eg-tensor} with ${\tt tol}_0 = 10^{-4} $, $ \hat \epsilon = 5\times 10^{-3} $, {$ \nu = 5\times 10^{3}$} and $ \kappa_c = 10^{3}$.}
	\label{table-tensor} \\
	\hline
	\texttt{Problem} & \texttt{Solver}    & \texttt{Iteration} & \texttt{Time} & \texttt{KKT} & \texttt{Objective} \\ \hline
	\endfirsthead
	\multicolumn{6}{c}%
	{{\bfseries Table \thetable\ continued from previous page}} \\
	\hline
	\texttt{Problem} & \texttt{Solver}    & \texttt{Iteration} & \texttt{Time} & \texttt{KKT} & \texttt{Objective} \\ \hline
	\endhead
	%
	
	\multirow{2}{*}{$\texttt{nonsym(6,4)}$} & \texttt{SDPLR} & 17, 859 &  0.8 & 1.4e-07, 0.0e+00, \textbf{1.1e-05} &  3.0767777e+00 \\
	& \texttt{ManiSDP} & 120, 600, 8022 &  5.2 & \textbf{4.7e-04}, 0.0e+00, \textbf{2.7e-05} &  3.0777468e+00 \\
	216 & \texttt{SDPNAL+} & 300, 4, 270 &  2.5 & 7.6e-08, 1.5e-08, 2.1e-15 &  3.0767784e+00 \\
	9260 & \texttt{Ours} & 932, 8, 513 &  7.5 & 2.0e-16, 8.8e-17, 5.1e-08 &  3.0767803e+00 \\ \hline
	\multirow{2}{*}{$\texttt{nonsym(7,4)}$} & \texttt{SDPLR} & 17, 527 &  0.8 & 6.7e-08, 0.0e+00, 6.3e-06 &  5.0740764e+00 \\
	& \texttt{ManiSDP} & 110, 550, 5894 &  6.6 & \textbf{4.5e-04}, 0.0e+00, 8.7e-06 &  5.0741640e+00 \\
	343 & \texttt{SDPNAL+} & 300, 5, 207 &  4.8 & 6.7e-08, 5.3e-08, 2.6e-15 &  5.0740771e+00 \\
	21951 & \texttt{Ours} & 1000, 6, 421 & 16.3 & 2.9e-16, 1.2e-13, 9.2e-07 &  5.0741640e+00 \\
\hline

\end{longtable}
}

\normalsize

\section{Concluding remarks}
\label{conclusions}
We have analyzed and implemented a squared smoothing Newton method {via} the Huber {smoothing} function for solving semidefinite programming problems (SDPs). With {a careful}  design of the algorithmic framework, our theoretical analysis has shown that the proposed algorithm is well-defined, guarantees global convergence and admits a superlinear convergence rate under the primal and dual constraint nondegeneracy conditions. Besides establishing these convergence properties, we have also conducted extensive numerical experiments on solving various classes of SDPs to evaluate the practical performance of our algorithm. We have compared our method with the state-of-the-art SDP \blue{solvers, including {\tt SDPLR}, {\tt ManiSDP} and {\tt SDPNAL+},} and the numerical results have demonstrated the excellent efficiency of our algorithm. We note that the current implementation of the algorithm is not {as mature as we would hope for} since the performance may depend sensitively on some parameters. However, given the promising numerical results on the {tested} examples, we are inspired to conduct a more robust implementation in a future work.

\section*{Acknowledgments}
The research of Defeng Sun is supported in part by the  Hong Kong Research Grants Council under grant 15307523, and the research of Kim-Chuan Toh is supported by the Ministry of Education, Singapore, under its Academic Research Fund Tier 3 grant call (MOE-2019-T3-1-010).

\blue{We would like to thank the editors and referees for  providing numerous valuable suggestions which have helped to improve the quality of this paper.}

\appendix
\section{Proof of Proposition \ref{prop-huber-properties}}\label{proof-huber-properties}
	Part 1 is a direct consequence of \cite[Proposition 4.3]{chen2003analysis}.	We next prove part 2. Since $h(\cdot,\cdot)$ is locally Lipschitz continuous on $\R\times \R$, then by \cite[Theorem 9.67]{rockafellar2009variational}, there exist continuously differentiable functions $h_\ell:\R\times \R \rightarrow \R$, $\ell \geq 1$, converging uniformly to $h$ and satisfying
	\[
	\norm{h_\ell'(\tau, \xi)} \leq L , \quad \forall (\tau, \xi) \in \cJ:= [\epsilon-\nu, \epsilon + \nu ] \times (\bigcup_{i = 1}^n [d_i-\nu_i, d_i + \nu_i] ) ,
	\]
	with some constants $L > 0$, $\nu > 0$ and $\nu_i > 0$. For any $ (\tau, H)\in \R\times \S^n $ with $H =\tilde P\mathrm{diag}(\tilde d_1, \dots, \tilde d_n)\tilde P^T$, define
	\[
	\Phi_\ell(\tau, H) = \tilde P\mathrm{diag}(h_\ell(\tau, \tilde d_1), \dots, h_\ell(\tau, \tilde d_n)) \tilde P^T,\quad \ell\geq 1.
	\]
	From \cite{horn2012matrix}, we may assume that there exists a neighborhood of $(\epsilon, W)$, denoted by $\cU$, such that $(\tau, \tilde d_i) \in \cJ$, for {all $i = 1,\dots, n$, and $ (\tau, H) \in \cU$}. Note that {$\Phi_\ell$} converges to $\Phi$ uniformly on $\cU$. {Fix $ (\tau_1,H_1) $, $ (\tau_2, H_2)\in \cU $ such that $ (\tau_1,H_1)\neq (\tau_2, H_2) $.} {Then, for any $\hat L > 0$ and $\ell$ sufficiently large, it holds that}
	\begin{equation}\label{lemma-Lip-Phi-2}
		\norm{\Phi_\ell(\tau, H) - \Phi(\tau, H)} \leq {\hat L \norm{(\tau_1, H_1) - (\tau_2, H_2)}},\quad \forall (\tau, H)\in \cU.
	\end{equation}
	Using \eqref{lemma-Lip-Phi-2}, for any $(\tau, H)\in \cU$, it follows that
	\[
	\begin{aligned}
		&\; \norm{\Phi(\tau_1, H_1) - \Phi(\tau_2, H_2)} \\
		\leq &\; \norm{\Phi(\tau_1, H_1) - \Phi_\ell(\tau_1, H_1)} + \norm{\Phi_\ell(\tau_1, H_1) - \Phi_\ell(\tau_2, H_2)} + \norm{\Phi_\ell(\tau_2, H_2) -  \Phi(\tau_2, H_2)} \\
		\leq &\; 2\hat L \norm{(\tau_1, H_1) - (\tau_2, H_2)}  + \norm{\int_0^1 \Phi_\ell'(\tau_1 + t(\tau_2-\tau_1), H_1 + t(H_2-H_1))(\tau_1 -\tau_2, H_2 - H_1) dt} \\
		\leq &\; (2\hat L + L)\norm{(\tau_1, H_1) - (\tau_2, H_2)}
	\end{aligned}
	\]
	for $\ell\geq 1$ sufficiently large. Since $ \hat L > 0 $ is arbitrary, we see that $ \Phi $ is locally Lipschitz continuous with modulus $ L>0 $.
	
	Then, we turn to prove part 3. Let us recall that $ \Phi(0,W) = \sum_{j = 1}^rh(0, \lambda_j)Q_j $ and write
	\[
	\begin{aligned}
		\Phi(t\tau, W + tH)
		= &\;  \sum_{j = 1}^rh(0, \lambda_j)Q_j(t)
		+ \sum_{i = 1}^n \left( h(t\tau, d_i(t)) - h(0, d_i)\right) p_i(t)p_i(t)^T.
	\end{aligned}
	\]
	Then, it follows that
	\[
	\begin{aligned}
		&\;\lim_{t\downarrow 0} \frac{1}{t}\left(\Phi(t\tau, W + tH) - \Phi(0, W) \right)\\
        = &\;\sum_{j = 1}^rh(0, \lambda_j)\lim_{t\downarrow 0} \frac{1}{t}\left(Q_j(t) - Q_j\right)  + \sum_{i = 1}^n \lim_{t\downarrow 0} \frac{1}{t}\left( h(t\tau, d_i(t)) - h(0, d_i)\right) p_i(t)p_i(t)^T.
	\end{aligned}
	\]
	By \cite[Eq. (2.9)]{shapiro2002differentiability}, for any $j = 1,\dots, r$, the following holds:
	\[
	   \lim_{t\downarrow 0} \frac{1}{t}\left(Q_j(t) - Q_j\right) = Q'(0)(H) = \frac{1}{2} \sum_{1\leq k\neq j\leq r} \frac{h(0, \lambda_k)-h(0, \lambda_j)}{\lambda_k- \lambda_j}\left(Q_jHQ_k + Q_kHQ_j\right).
	\]
	From \cite[Proposition 3.1]{shapiro2002differentiability}, for any $i = 1,\dots, n$, we know that $d_i'(W;H)$ is well-defined, and it holds that
	\[
	\begin{aligned}
		&\; \lim_{t\downarrow 0} \frac{1}{t}\left( h(t\tau, d_i(t)) - h(0, d_i)\right)= h'((0, d_i); (\tau, d_i'(W; H)))\\
		= &\;   \sum_{i\in \alpha} \left( d_i'(W; H) - \frac{\abs{\tau}}{2}\right)p_ip_i^T + \sum_{i\in \beta} h(\tau, d_i'(W;H))p_ip_i^T,
	\end{aligned}
	\]
	where we have used the fact that $h$ is directionally differentiable with
	\[
	h'((0, u); (\tau, v)) = \left\{
	\begin{array}{ll}
		v - \frac{\abs{\tau}}{2} & u > 0 \\
		h(\tau, v) & u = 0 \\
		0 & u < 0
	\end{array}
	\right.,\quad \forall (\tau, v)\in \R\times \R.
	\]
	Since $p_i(t)\rightarrow p_i$ as $t\downarrow 0$, $ \displaystyle \lim_{t\downarrow 0} \frac{1}{t}\left(\Phi(t\tau, W + tH) - \Phi(0, W) \right) $
	exists. Hence, $\Phi$ is directionally differentiable at $(0, W)$, and \eqref{eq-dir-diff-1} holds true. Finally, the explicit expression of the directional differential $d_i'(W;H)$ ($ 1\leq i \leq n$) is again obtained from \cite[Proposition 3.1]{shapiro2002differentiability}.
	
	Finally, we prove part 4. {Based on part 1, $ \Phi $ is continuously differentiable at any $ (\epsilon, W) $ with $ \epsilon \neq 0 $. Hence, $\Phi$ is naturally strongly semismooth at these points. Thus, we only need to show that $ \Phi $ is strongly semismooth at $ (0, W) $ for any $ W\in \S^n $.} We have already known that $\Phi$ is locally Lipschitz continuous and directionally differentiable everywhere. By \cite[Theorem 3.7]{sun2002semismooth}, we only need to show that for any $(\tau, H)\in \R\times \S^n$ with $\norm{(\tau, H)} \rightarrow 0$,
	\begin{equation}\label{eq-semismooth-2}
		\Phi(\tau, W+H) - \Phi(0, W) - \Phi'((\tau, W+H); (\tau, H)) = O\left(\norm{(\tau, H)}^2\right).
	\end{equation}
	First, similar to the proof of part 3, one can show for any $ (\tau, H)\in \R\times \S^n $,
	\begin{equation}\label{eq-semismooth-1}
		\Phi'((\tau, W+H); (\tau, H))
			=   \sum_{j = 1}^rh(\tau, \lambda_j(1))Q_j'(1)(H)   + \sum_{i = 1}^n h'((\tau, d_i(1)); (\tau, d_i'(W+H; H)))p_i(1)p_i(1)^T.
	\end{equation}
	Then, by the fact that $d_i$ is strongly semismooth everywhere (see for example \cite[Proposition 3.2]{shapiro2002differentiability}), we deduce that
	\begin{eqnarray*}
		& & \Phi(\tau, W+H) - \Phi(0, W) \\
		&=&  \sum_{j=1}^rh(\tau,\lambda_j(1))\left(Q_j(1) - Q_j\right) + \sum_{i = 1}^n \left(h(\tau, d_i(1)) - h(0, d_i)\right)p_ip_i^T \\
		&=& \sum_{j=1}^rh(\tau,\lambda_j(1))Q_j'(1)(H) + O\left(\norm{H}^2\right)  +\sum_{i = 1}^n \left(h'((\tau, d_i(1)); (\tau, d_i'(W+H; H)))\right)p_ip_i^T \\
		&& + O\left(\norm{(\tau, H)}^2\right) \\
		&=& \sum_{j=1}^rh(\tau,\lambda_j(1))Q_j'(1)(H) +\sum_{i = 1}^n \left(h'((\tau, d_i(1)); (\tau, d_i'(W+H; H)))\right)p_i(1)p_i(1)^T \\
		&&  + O\left(\norm{(\tau, H)}^2\right),
	\end{eqnarray*}
	which together with \eqref{eq-semismooth-1} and the fact that $p_i$ is analytic around $W$ implies \eqref{eq-semismooth-2}. Thus, the proof is completed.

\section{Proof of Lemma \ref{lemma-partial-Gamma}}\label{proof-partial-Gamma}
First, let $V\in \partial_B\Phi(0,W)$. By the definition of $\partial_B\Phi(0,W)$, there exists a sequence $\{(\epsilon^k, W^k)\}$ converging to $(0,W)$ with $\epsilon^k\neq 0$ such that $V = \displaystyle \lim_{k\rightarrow\infty}\Phi'(\epsilon^k,W^k)$. Let each $W^k$ have the following spectral decomposition:
\[
W^k = P^k D^k (P^k)^T, \quad D^k := \mathrm{diag}\left(D_\alpha^k, D_\beta^k, D_\gamma^k\right),
\]
where $D_\alpha^k =  \mathrm{diag}(d_1^k,\dots, d_{\abs{\alpha}}^k)$, $D_\beta^k = \mathrm{diag}(d_{\abs{\alpha}+1}^k,\dots, d_{\abs{\alpha} + \abs{\beta}}^k)$, and $D_\gamma^k = \mathrm{diag}(d_{\abs{\alpha} + \abs{\beta}+1}^k,\dots, d_{n}^k)$,
with $d_1^k \geq \dots \geq d_n^k$. For simplicity, denote $d^k = (d_1^k,\dots,d_n^k)^T\in \R^n$. By taking {a subsequence} if necessary, we {may assume} without loss of generality {that} (a) $\displaystyle\lim_{k\rightarrow\infty}D^k = D$, $\displaystyle \lim_{k\rightarrow\infty}P^k = P$, and (b) both sequences $\{\Omega(\epsilon^k,d^k)\}$ and $\{\cD(\epsilon^k, d^k)\}$ converge, {where $ \Omega(\epsilon^k, d^k) $ and $ \cD(\epsilon^k, d^k) $ are defined as in \eqref{eq-huber}}. In particular, we see that $\{\Omega(\epsilon^k,d^k)_{\alpha\alpha}\} $, $\{\Omega(\epsilon^k,d^k)_{\alpha\beta}\} $ converge {to} two matrices of all ones of suitable sizes, respectively, $\{\Omega(\epsilon^k,d^k)_{\alpha\gamma}\} $ converges to $\Omega_0(d)_{\alpha\gamma}$ and the limit of the sequence $\{\Omega(\epsilon^k,d^k)_{\beta\beta}\} $ exists. Moreover, let
\[
	\cD(\epsilon^k,d^k) = \mathrm{diag}(\cD(\epsilon^k,d^k)_\alpha, \cD(\epsilon^k,d^k)_\beta,\cD(\epsilon^k,d^k)_\gamma).
\]
It holds that $\{\cD(\epsilon^k,d^k)_\gamma\}$ converges to the zero matrix, and the limits of $\{\cD(\epsilon^k,d^k)_\alpha\}$ and $\{\cD(\epsilon^k,d^k)_\beta\}$ exist.

From Proposition \ref{prop-huber-properties} and $ \epsilon_k\neq 0 $, for any $(\tau, H)\in \R\times \S^n$, we see that
\begin{equation}
	\label{eq-Gamma-1}
	\Phi'(\epsilon^k,W^k)(\tau, H) = P^k\left[\Omega(\epsilon^k, d^k)\circ \tilde H^k + \tau \cD(\epsilon^k, d^k)\right](P^k)^T, 
\end{equation}
where $\tilde H_k:= (P^k)^THP^k$. Taking \blue{the limit} of both sides in \eqref{eq-Gamma-1} yields that
\[		\begin{aligned}
V(\tau, H) 
 = &\;P\begin{pmatrix}
		\tilde H_{\alpha\alpha} & \tilde H_{\alpha\beta} & \Omega_0(d) \circ \tilde H_{\alpha\gamma} \\
		\tilde H_{\alpha\beta}^T & 	\displaystyle
		\lim_{k\rightarrow\infty}\Omega(\epsilon^k, d^k)_{\beta\beta}\circ \tilde H_{\beta\beta} & 0 \\
		\tilde H_{\alpha\gamma}^T \circ \Omega_0(d)^T & 0 & 0
	\end{pmatrix}P^T \\
	&\; + \tau P\begin{pmatrix}
		\displaystyle \lim_{k\rightarrow\infty} \cD(\epsilon^k, d^k)_\alpha & 0 & 0 \\
		0 & 	\displaystyle
		\lim_{k\rightarrow\infty}\cD(\epsilon^k,d^k)_\beta & 0 \\
		0 & 0 & 0
	\end{pmatrix}P^T.
\end{aligned}
\]

For each $k\geq 1$, define $Y^k:= P \mathrm{diag}(0, D_\beta^k, 0)P^T$ and $\tilde Y^k := P^TY^kP$. Note that the mapping $\cL$ is F-differentiable at $(\epsilon^k, Y^k)$ since $ \epsilon_k\neq 0 $. For $k$ sufficiently large, we have from \eqref{eq-dir-diff-new} and Proposition \ref{prop-huber-properties} that
\[
\begin{aligned}
	&\; \cL'(\epsilon^k, Y^k)(\tau, H) =  \lim_{t\downarrow0}\frac{ \cL(\epsilon^k+t\tau, Y^k + tH) - \cL(\epsilon^k, Y^k)}{t} \\
		=&\; P \begin{pmatrix}
		\tilde H_{\alpha\alpha}  & \tilde H_{\alpha\beta} & [\Omega_0(d)]_{\alpha\gamma}\circ \tilde H_{\alpha\gamma} \\
		\tilde H_{\alpha\beta}^T & \Phi'_{\abs{\beta}}(\epsilon^k, D_\beta^k)(\tau, \tilde H_{\beta\beta}) & 0 \\
		[\Omega_0(d)]_{\alpha\gamma}^T \circ H_{\alpha\gamma}^T & 0 & 0
	\end{pmatrix} P^T -  \frac{\tau}{2}\mathrm{sgn}(\epsilon^k)\sum_{i\in\alpha}p_ip_i^T \\
	= &\; P \begin{pmatrix}
		\tilde H_{\alpha\alpha} & \tilde H_{\alpha\beta} & [\Omega_0(d)]_{\alpha\gamma}\circ \tilde H_{\alpha\gamma} \\
		\tilde H_{\alpha\beta}^T & \Omega(\epsilon^k,d^k)_{\beta\beta}\circ \tilde H_{\beta\beta} & 0 \\
		[\Omega_0(d)]_{\alpha\gamma}^T \circ H_{\alpha\gamma}^T & 0 & 0
	\end{pmatrix} P^T \\
	&\; + \tau P \begin{pmatrix}
		\cD(\epsilon_k, d^k)_\alpha & 0 & 0 \\
		0 & \cD(\epsilon^k,d^k)_\beta & 0 \\
		0 & 0 & 0
	\end{pmatrix}P^T,
\end{aligned}
\]	
which implies that $V(\tau,H) = \displaystyle\lim_{k\rightarrow\infty}\cL'(\epsilon^k, Y^k)(\tau, H)$. Hence, $V\in \partial_B\cL(0,0)$.

Conversely, choose $V\in \partial_B\cL(0,0)$. By definition, there exists a sequence $\{(\epsilon^k, Y^k)\}$ converging to $(0,0)$ with $\epsilon^k\neq 0$ such that $V = \displaystyle\lim_{k\rightarrow\infty}\cL'(\epsilon^k, Y^k)$. Let $\tilde Y^k := P^TY^kP$. Assume that $ \tilde Y^k_{\beta\beta}$ has the spectral decomposition: $\tilde Y^k_{\beta\beta} = U^k\tilde D^k_\beta (U^k)^T$, where $\tilde D^k_\beta = \mathrm{diag}(\tilde z^k)$, $\tilde z^k := (\tilde z_1^k,\dots, \tilde z_{\abs{\beta}}^k)^T\in \R^{\abs{\beta}}$, $\tilde z_1^k\geq \dots \geq \tilde z_{\abs{\beta}}^k$, and $U^k\in \R^{\abs{\beta}\times \abs{\beta}}$ is orthogonal. Then, for any $(\tau,H)\in \R\times \S^n$ with $\tilde H := P^THP$, we get from \eqref{eq-dir-diff-new} and Proposition \ref{prop-huber-properties} that
\[
    \cL'(\epsilon^k, Y^k)(\tau, H) = P \begin{pmatrix}
		\tilde H_{\alpha\alpha} & \tilde H_{\alpha\beta} & [\Omega_0(d)]_{\alpha\gamma}\circ \tilde H_{\alpha\gamma} \\
		\tilde H_{\alpha\beta}^T & \hat H_{\beta\beta} & 0 \\
		[\Omega_0(d)]_{\alpha\gamma}^T \circ H_{\alpha\gamma}^T & 0 & 0
	\end{pmatrix} P^T  -  \frac{\tau}{2}\mathrm{sgn}(\epsilon^k)\sum_{i\in\alpha}p_ip_i^T,
\]
where $\hat H_{\beta\beta}:= U^k\left[\Omega(\epsilon^k, \tilde z^k)\circ \left((U^k)^T\tilde H_{\beta\beta} U^k \right) + \tau \cD(\epsilon^k,\tilde z^k)  \right] (U^k)^T$. For any $k\geq 1$, define
\[
W^k = W + P\begin{pmatrix}
	0 & 0 & 0 \\
	0 & \tilde Y^k_{\beta\beta} & 0 \\
	0 & 0 & 0
\end{pmatrix}P^T,\quad \tilde W^k = P^TW^kP = \begin{pmatrix}
	D_\alpha & 0 & 0 \\
	0 & \tilde Y^k_{\beta\beta}  & 0 \\
	0 & 0 & D_\gamma
\end{pmatrix}
\]
where
\[
D_\alpha = \mathrm{diag}(d_1,\dots, d_{\abs{\alpha}}), \quad D_\gamma = \mathrm{diag}(d_{\abs{\alpha} + \abs{\beta}+1},\dots, d_{n}).
\]
Moreover, we partition $P$ as $P = (P_\alpha,P_\beta,P_\gamma)$, set $P^k = (P_\alpha,P_\beta U^k,P_\gamma)$ and construct a vector $d^k\in \R^n$ as follows
\[
d^k_i = \left\{
\begin{array}{ll}
	d_i & i\in \alpha\cup \gamma \\
	\tilde z^k_{i - \abs{\alpha}} & i\in \beta
\end{array}.
\right.
\]	
Clearly, it holds that $W^k = P^k \mathrm{diag}(d^k)(P^k)^T$. Since $\epsilon^k\neq 0$, $\Phi$ is F-differentiable at $(\epsilon^k, W^k)$, and
\[
\Phi'(\epsilon^k, W^k)(\tau, H) = P^k\left[\Omega(\epsilon^k, d^k)\circ\left((P^k)^THP^k\right) + \tau \cD(\epsilon^k, d^k)\right](P^k)^T.
\]
Let us now assume without loss of generality that the three sequences $\{U^k\}$, $\{\Omega(\epsilon^k, d^k)\}$ and $\{\cD(\epsilon, d^k)\}$ converge (since they are all uniformly bounded). Then, simple {calculations show} that
\[
\lim_{k\rightarrow\infty}[\Omega(\epsilon^k, d^k)]_{ij} = \left\{
\begin{array}{ll}
	1 & i\in \alpha, j\in \alpha\cup \beta \\
	\Omega_0(d) & i\in \alpha, j\in \gamma \\
	0 & i\in \beta\cup \gamma, j\in \gamma \\
	\displaystyle[\lim_{k\rightarrow\infty}\Omega(\epsilon^k, \tilde z^k)]_{(i-\abs{\alpha})(j-\abs{\alpha})} & i\in \beta,j\in \beta
\end{array},
\right.
\]
and that
\[
\lim_{k\rightarrow\infty}\cD(\epsilon^k, d^k) = \begin{pmatrix}
	\displaystyle \lim_{k\rightarrow\infty}\cD(\epsilon^k,d_\alpha) & 0 & 0 \\
	0 & \displaystyle \lim_{k\rightarrow \infty}\cD(\epsilon^k, \tilde z^k) & 0 \\
	0 & 0 & 0
\end{pmatrix},
\]
where $d_\alpha := (d_1,\dots, d_{\abs{\alpha}})^T\in\R^{\abs{\alpha}}$. As a consequence, we get
\[
\lim_{k\rightarrow\infty} (P^k)^T \left( \cL'(\epsilon^k, Y^k)(\tau, H) - \Phi'(\epsilon^k, W^k)(\tau, H)\right) P^k = 0, \quad \forall (\tau, H)\in \R\times \S^n,
\]
which further implies that $V(\tau, H) = \displaystyle\lim_{k\rightarrow\infty}\Phi'(\epsilon^k, W^k)(\tau, H)$, for all $(\tau, H)\in \R\times \S^n$. Then, $V\in \partial_B\Phi(0, W)$. Therefore, the proof is completed.

\section{Proof of Lemma \ref{lemma-partial-B-Phi}} \label{proof-partial-B-Phi}
We only prove the case for {the} B-subdifferential as the case for Clarke's generalized Jacobian can be proved \blue{similarly}.
	Let the smooth mapping ${\Psi:\R\times \S^n\rightarrow \R\times \S^n}$ be defined as $\Psi(\tau, H):= (\tau, P^THP)$ for any $(\tau, H)\in \R\times \S^n$. It is clear that $ \Psi'(\tau, H) : \R\times \S^n \rightarrow \R\times \S^n$ is onto. Define the mapping $\Upsilon: \R\times \S^n \rightarrow \S^n$ as follows:
	\[
		\begin{aligned}
			\Upsilon(\nu, Y) := &\; P
	\begin{pmatrix}
		Y_{\alpha\alpha} & Y_{\alpha\beta} & [\Omega_0(d)]_{\alpha\gamma} \circ Y_{\alpha\gamma} \\
		Y_{\alpha\beta}^T & \Phi_{\abs{\beta}}(\nu, Y_{\beta\beta}) & 0 \\
		Y_{\alpha\gamma}^T\circ [\Omega_0(d)]_{\alpha\gamma}^T & 0 & 0
	\end{pmatrix}
	P^T - \frac{\abs{\nu}}{2}\sum_{i\in\alpha}p_ip_i^T,  
		\end{aligned}
	\]
	$\forall (\nu, Y)\in \R\times \S^n.$ Then, by \eqref{eq-dir-diff-new}, it holds that $\cL(\tau, H) = \Upsilon(\Psi(\tau, H))$. Now, by \cite[Lemma 1]{chan2008constraint} and Lemma \ref{lemma-partial-Gamma}, we have
	\[
	\partial_B\Phi(0,W) = \partial_B\cL(0,0) = \partial_B\Upsilon(0,0)\Psi'(0, 0),
	\]
	which proves the first part of the lemma. The expression of $V_{\abs{\beta}}\in \partial_B\Phi(0,0)$ follows directly from its definition and Proposition \ref{prop-huber-properties}. Thus, the proof is completed.

 \section{Proof of Lemma \ref{lemma-V-PSD}} \label{proof-V-PSD}
 We first consider $ V\in \partial_B\Phi(0,W) $. Using Lemma \ref{lemma-partial-B-Phi}, there exist an orthogonal matrix $ U\in \R^{\abs{\beta}\times \abs{\beta}} $ and a symmetric matrix $ \Omega_{\abs{\beta}}\in\S^{\abs{\beta}} $ such that
	\[
		\begin{aligned}
			&\; \inprod{H - V\left(0, H\right)}{V\left(0, H\right)}\\
			= &\; \inprod{P^T\left(H - V(0, H)\right)P}{P^TV(0, H)P}
			 \\
			= &\; 2\inprod{\left(E_{\alpha\gamma} -  [\Omega_0(d)]_{\alpha\gamma}\right)\circ \tilde H_{\alpha\gamma}}{ [\Omega_0(d)]_{\alpha\gamma}\circ \tilde H_{\alpha\gamma}} \\
			&\;  + \inprod{\tilde H_{\beta\beta} - U\left[\Omega_{\abs{\beta}}\circ\left(U^T\tilde H_{\beta\beta}U\right)\right]U^T}{U\left[\Omega_{\abs{\beta}}\circ\left(U^T\tilde H_{\beta\beta}U\right)\right]U^T} \\
			= &\; 2\inprod{(E_{\alpha\gamma} -  [\Omega_0(d)]_{\alpha\gamma})\circ \tilde H_{\alpha\gamma}}{ [\Omega_0(d)]_{\alpha\gamma}\circ \tilde H_{\alpha\gamma}} \\
			&\; + \inprod{\left(E_{\beta\beta} - \Omega_{\abs{\beta}}\right)\circ\left(U^T\tilde H_{\beta\beta}U\right)}{\Omega_{\abs{\beta}}\circ\left(U^T\tilde H_{\beta\beta}U\right)},
		\end{aligned}
	\]
	where $ E_{\alpha\gamma} $ and $ E_{\beta\beta} $ denote the matrices of all ones in $ \R^{\abs{\alpha}\times \abs{\gamma}} $ and $ \R^{\abs{\beta}\times\abs{\beta}} $, respectively. Notice that the elements in both the matrices $ \Omega_0(d) $ and $ \Omega_{\abs{\beta}} $ are all inside the interval $ [0,1] $. Hence, we conclude that $ \inprod{H - V(0, H)}{V(0, H)} \geq 0 $, for any $ V\in \partial_B\Phi(0,W) $ and $ H\in \S^n $.
	
	Next, we let $ V\in \Phi(0,W) $. By Carath\'{e}odory's theorem \cite{caratheodory1911variabilitatsbereich}, there exist a positive integer $ q $ and $ V^i\in \partial_B\Phi(0,W) $, $ i = 1,\dots,q $, such that $ V  $ is the convex combination of $ V^1,\dots, V^q $. Therefore, there exist nonnegative scalars $ t_1,\dots, t_q $ such that $ V = \sum_{i = 1}^q t_iV^i $ with $ \sum_{i = 1}^qt_i = 1 $. Define the convex function $ \theta(X) = \inprod{X}{X},X\in \S^n $. By {the convexity of $\theta$}, we have
	\[
		\begin{aligned}
			\inprod{V(0, H)}{V(0,H)} = &\; \theta(V(0, H))
		=  \theta\left(\sum_{i = 1}^q t_iV^i(0, H)\right) \leq  \sum_{i = 1}^q t_i\theta(V^i(0,H)) = \inprod{H}{V(0,H)}.
		\end{aligned}
	\]
	{This completes the proof.}

\section{Proof of Lemma \ref{lemma-nonsingular}}\label{proof-nonsingular}
    Simple {calculations give}
	\[
	\begin{aligned}
		&\; \hcE'(\epsilon,  X, y, Z) = 
	\begin{pmatrix}
			1 & 0 & 0 & 0 \\
				{\kappa_p y} & \cA & \kappa_p{\epsilon}I_m & 0 \\
				0 & 0 & -\cA^* & -\cI \\
				{\kappa_c X}- {\Phi_1'(\epsilon, X- Z)} & (1+\kappa_c{\epsilon})\cI - \Phi_2'(\epsilon, X-Z) & 0 & {\Phi_2'(\epsilon, X-Z)}
			\end{pmatrix},
	\end{aligned}
	\]
	where {$ \Phi'_1 $ and $ \Phi_2' $ denote the partial derivatives of $ \Phi $ with respect to the first and second arguments, respectively, and} $\cI$ is the identity map over $\S^n$. To show that $\hcE'(\epsilon, X, y, Z)$ is nonsingular, it suffices to show that the following system of linear equations
	\[
	\hcE'(\epsilon, X, y, Z)(\Delta \epsilon, \Delta X, \Delta y, \Delta Z) = 0,\quad (\Delta \epsilon, \Delta X, \Delta y, \Delta Z)\in \R\times \X,
	\]
	has only the trivial {solution $(\Delta \epsilon, \Delta X, \Delta y, \Delta Z) = (0, 0, 0, 0)$}. It is obvious that $\Delta \epsilon = 0$. Since $(1+\kappa_c{\epsilon})\Delta X - \Phi_2'(\epsilon, X-Z)\Delta X {+} \Phi_2'(\epsilon, X-Z)\Delta Z = 0$, it follows that
	\begin{equation}
		\label{lemma-nonsingular-1}
		\Delta X = {-}\left((1+\kappa_c{\epsilon})\cI - \Phi_2'(\epsilon, X-Z)\right)^{-1} \Phi_2'(\epsilon, X-Z)\Delta Z.
	\end{equation}
	By the equality $-\cA^*\Delta y - \Delta Z = 0$, it follows that $\Delta Z = -\cA^*\Delta y$ which together with \eqref{lemma-nonsingular-1} implies that
	\begin{equation}\label{lemma-nonsingular-2}
		\Delta X = \left((1+\kappa_c{\epsilon})\cI - \Phi_2'(\epsilon, X-Z)\right)^{-1} \Phi_2'(\epsilon, X-Z)\cA^*\Delta y.
	\end{equation}
	{Using} the fact that $\cA \Delta X + \kappa_p{\epsilon}\Delta y= 0$ {and \eqref{lemma-nonsingular-2}, we get}
	\[
	\left( \kappa_p{\epsilon}I_m  + \cA \left((1+\kappa_c{\epsilon})\cI - \Phi_2'(\epsilon, X-Z)\right)^{-1} \Phi_2'(\epsilon, X-Z)\cA^*\right) \Delta y = 0.
	\]
	{Since $ 0 \preceq \Phi_2'(\epsilon, X-Z) \preceq \cI $, it is easy to show that the coefficient matrix in the above linear system is symmetric positive definite. Thus,} $\Delta y = 0$, which further implies that $\Delta X = \Delta Z = 0$. {Therefore}, the proof is completed.

\section{Proof of Lemma \ref{lemma-step-size}} \label{proof-step-size}
Since $ \epsilon'\in \R_{++} $, by Lemma \ref{lemma-nonsingular}, $ \hcE'(\epsilon', X', y', Z') $ is nonsingular. Since $ \hcE $ is continuously differentiable around $ (\epsilon', X', y', Z') $, there exists an open neighborhood $ \cU $ of $ (\epsilon', X', y', Z') $ such that for any $ (\epsilon, X, y, Z) $, $ \hcE'(\epsilon, X, y, Z) $ is nonsingular. {Therefore, the existence of $\bDelta$ is guaranteed.}
	
	Denote $\cR(\epsilon, X, y, Z) := \cE(\epsilon, X, y, Z) + \cE'(\epsilon, X, y, Z) {\bDelta}$.
	Then one can verify that $(\Delta \epsilon, \Delta X, \Delta y, \Delta Z)$ is the unique solution of the following equation
	\[
	\hcE(\epsilon, X, y, Z) +  \hcE'(\epsilon, X, y, Z) {\bDelta}
	=
	\begin{pmatrix}
		\zeta(\epsilon, X, y, Z) \hat \epsilon \\ \cR(\epsilon, X, y, Z)
	\end{pmatrix}.
	\]
	Hence, it holds that
	\begin{equation}\label{eq-step-size-1}
		\begin{aligned}
			&\; \inprod{\nabla \psi(\epsilon, X, y, Z)}{{\bDelta}
			} = \inprod{2\nabla \hcE(\epsilon, X, y, Z)\hcE(\epsilon, X, y, Z)}{{\bDelta}
			} \\
			= &\; \inprod{2\hcE(\epsilon, X, y, Z)}{\begin{pmatrix}
					\zeta(\epsilon, X, y, Z) \hat \epsilon \\ \cR(\epsilon, X, y, Z)
				\end{pmatrix} - \hcE(\epsilon, X, y, Z)} \\
			= &\; -2\psi(\epsilon, X, y, Z) + 2\zeta(\epsilon, X, y, Z)\epsilon\hat \epsilon + 2\inprod{\cR(\epsilon, X, y, Z)}{\cE(\epsilon, X, y, Z)} \\
			\leq &\; -2\psi(\epsilon, X, y, Z) + 2r\min\{1, \psi(\epsilon, X, y, Z)^{(1+\tau)/2}\}\epsilon\hat \epsilon  + 2{\hat\eta} \psi(\epsilon, X, y, Z)^{1/2}\norm{\cE(\epsilon, X, y, Z)}.
		\end{aligned}
	\end{equation}
	We consider two possible cases when $\psi(\epsilon, X, y, Z) > 1$ and $\psi(\epsilon, X, y, Z) \leq 1$.
	
	If $\psi(\epsilon, X, y, Z) > 1$, {then} \eqref{eq-step-size-1} implies that
	\[
	\begin{aligned}
		&\; \inprod{\nabla \psi(\epsilon, X, y, Z)}{{\bDelta}
		} \\
		\leq &\; -2\psi(\epsilon, X, y, Z) + 2r\epsilon\hat \epsilon + 2{\hat\eta} \psi(\epsilon, X, y, Z)^{1/2}\sqrt{\psi(\epsilon, X, y, Z) - \epsilon^2} \\
		\leq &\; -2\psi(\epsilon, X, y, Z) + 2 \max\{r\hat \epsilon, {\hat\eta}\}\left(\epsilon + \psi(\epsilon, X, y, Z)^{1/2}\sqrt{\psi(\epsilon, X, y, Z) - \epsilon^2} \right) \\
		\leq &\; -2\psi(\epsilon, X, y, Z) + 2 \max\{r\hat \epsilon, {\hat\eta}\} \psi(\epsilon, X, y, Z) \\
		= &\; 2\left( \sqrt{2} \max\{r\hat \epsilon, {\hat\eta}\}  - 1 \right) \psi(\epsilon, X, y, Z).
	\end{aligned}
	\]
	
	If $\psi(\epsilon, X, y, Z) \leq 1$, {then} \eqref{eq-step-size-1} implies that
	\[
	\begin{aligned}
		&\; \inprod{\nabla \psi(\epsilon, X, y, Z)}{{\bDelta}
		} \\
		\leq &\; -2\psi(\epsilon, X, y, Z) + 2r\psi(\epsilon, X, y, Z)^{(1+\tau)/2}\epsilon\hat \epsilon  + 2{\hat\eta} \psi(\epsilon, X, y, Z)^{1/2}\sqrt{\psi(\epsilon, X, y, Z) - \epsilon^2} \\
		\leq &\; -2\psi(\epsilon, X, y, Z) + 2r\psi(\epsilon, X, y, Z)\hat \epsilon + 2{\hat\eta} \psi(\epsilon, X, y, Z)^{1/2}\sqrt{\psi(\epsilon, X, y, Z) - \epsilon^2} \\
		\leq &\; -2\psi(\epsilon, X, y, Z)  + 2 \max\{r\hat \epsilon, {\hat\eta}\}\psi(\epsilon, X, y, Z)^{1/2}\left(\epsilon\psi(\epsilon, X, y, Z)^{1/2} + \sqrt{\psi(\epsilon, X, y, Z) - \epsilon^2} \right) \\
		\leq &\; -2\psi(\epsilon, X, y, Z) + 2 \max\{r\hat \epsilon, {\hat\eta}\} \psi(\epsilon, X, y, Z) \\
		= &\; 2\left( \sqrt{2} \max\{r\hat \epsilon, {\hat\eta}\}  - 1 \right) \psi(\epsilon, X, y, Z).
	\end{aligned}
	\]
	
	For both cases, we always have that
	\begin{equation}\label{eq-step-size-2}
		\inprod{\nabla \psi(\epsilon, X, y, Z)}{{\bDelta}
		}  \leq 2\left( \sqrt{2} \max\{r\hat \epsilon, {\hat\eta}\}  - 1 \right) \psi(\epsilon, X, y, Z).
	\end{equation}
	
	Since $\nabla \psi(\cdot)$ is uniformly continuous on $\cU$, for all $(\epsilon, X, y, Z)\in \cU$ with $\epsilon > 0$, we have from the Taylor expansion that
	\begin{equation}\label{eq-step-size-3}
		\psi(\epsilon + \alpha \Delta \epsilon, X + \alpha \Delta X, y + \alpha \Delta y, Z + \alpha \Delta Z) = \psi(\epsilon, X, y, Z) + \alpha \inprod{\nabla \psi(\epsilon, X, y, Z)}{{\bDelta}
		} + o(\alpha).
	\end{equation}
	Combining \eqref{eq-step-size-2} and \eqref{eq-step-size-3}, it holds that
	\[
	\begin{aligned}
		 \psi(\epsilon + \alpha \Delta \epsilon, X + \alpha \Delta X, y + \alpha \Delta y, Z + \alpha \Delta Z) 
		 = &\; \psi(\epsilon, X, y, Z) + 2\alpha (\delta -1)\psi(\epsilon, X, y, Z) + o(\alpha) \\
		= &\;  \left[1 - 2\sigma(1 - \delta)\alpha \right] \psi(\epsilon, X, y, Z) + o(\alpha),
	\end{aligned}
	\]
	and the proof is completed.

\section{Proof of Lemma \ref{lemma-lb-epsilon}} \label{proof-lb-epsilon}
Since $(\epsilon^k, X^k, y^k, Z^k)\in \cN$, we have directly from the definition of $\cN$ that $\epsilon^k \geq \zeta(\epsilon^k, X^k, y^k, Z^k)\hat \epsilon$. Recall that $\Delta \epsilon^k = -\epsilon^k + \zeta(\epsilon^k, X^k, y^k, Z^k) \hat \epsilon  \leq 0$. It holds that
	\[
	\begin{aligned}
		&\; \epsilon^k + \alpha \Delta \epsilon^k  - \zeta(\epsilon^k + \alpha \Delta \epsilon^k, X^k + \alpha \Delta X^k,  y^k + \alpha \Delta y^k, Z^k + \alpha \Delta Z^k) \hat \epsilon \\
		\geq &\; \epsilon^k + \Delta \epsilon^k - \zeta(\epsilon^k + \alpha \Delta \epsilon^k, X^k + \alpha \Delta X^k, y^k + \alpha \Delta y^k, Z^k + \alpha \Delta Z^k)\hat \epsilon \\
		= &\;\zeta(\epsilon^k, X^k, y^k, Z^k)\hat \epsilon - \zeta(\epsilon^k + \alpha \Delta \epsilon^k, X^k + \alpha \Delta X^k, y^k + \alpha \Delta y^k, Z^k + \alpha \Delta Z^k)\hat \epsilon \\
		\geq &\; 0,
	\end{aligned}
	\]
	which indeed implies that $(\epsilon^k + \alpha \Delta \epsilon^k, X^k + \alpha \Delta X^k, y^k + \alpha \Delta y^k, Z^k + \alpha \Delta Z^k) \in \cN$. This completes the proof.

\section{Proof Theorem \ref{thm-global-convergence}} \label{proof-global-convergence}
It follows from Lemma \ref{lemma-step-size} and Lemma \ref{lemma-lb-epsilon} that Algorithm \ref{alg-ssnm} is well-defined and generates an infinite sequence containing in $ \cN $. 
	Since the line-search scheme is well-defined, it is obvious that
	\[
	\psi(\epsilon^{k+1}, X^{k+1}, y^{k+1}, Z^{k+1}) < \psi(\epsilon^k, X^k, y^k, Z^k),\quad \forall \; k\geq 0.
	\]
	The monotonically decreasing property of the sequence $\{\psi(\epsilon^k, X^k, y^k, Z^k)\}$ then implies that $\{\zeta_k\}$ is also {monotonically} decreasing. Hence, there exist $\bar \psi$ and $\bar \zeta$ such that
	\[
	\psi(\epsilon^k, X^k, y^k, Z^k) \rightarrow \bar \psi,\quad \zeta_k\rightarrow \bar \zeta,\quad k\rightarrow \infty.
	\]

	Note that $ \bar{\psi} \geq 0 $. We now prove that $ \bar \psi = 0 $ by contradiction. To this end, we assume that $\bar \psi > 0$, hence, $\bar \zeta > 0$. By the fact that the sequence $ \{(\epsilon^k, X^k, y^k, Z^k)\} $ is contained in $\mathcal{N}$, we see that there exists an $\epsilon' > 0$ such that $\epsilon^k > \epsilon' > 0$ for all $k\geq 0$. Also note that $\epsilon^k \leq \psi(\epsilon^k, X^k, y^k, Z^k)^{1/2} \leq \psi(\epsilon^0, X^0, y^0, Z^0)^{1/2}$. Therefore, there exists an $\epsilon'' > \epsilon'$ such that $\epsilon^k \in[\epsilon', \epsilon'']$ for all $k\geq 0$. We next claim that $\{( X^k, y^k, Z^k)\}$ is bounded. Let 
	\[
	R_1^k:= \mathcal{A}X^k + \kappa_p\epsilon^ky^k - b, \quad R_2^k:=-\mathcal{A}^*y^k - Z^k + C,\quad R_3^k:=(1+\kappa_c\epsilon^k)X^k - \Phi(\epsilon^k, X^k - Z^k),\quad k\geq 0.
	\]
	Since $\{\psi(\epsilon^k, X^k, y^k, Z^k)\}$ is bounded, we have that $\{R_1^k\}$, $\{R_2^k\}$ and $\{R_3^k\}$ are bounded. Using the first two equalities in the above, $Z^k$ can be expressed as follows:
	\[
	Z^k = -\mathcal{A}^*y^k + C - R_2^k 
	= \frac{1}{\kappa_p \epsilon^k}\mathcal{A}^*\mathcal{A}X^k + C^k,
	\]
	where $\left\{C^k := -\frac{1}{\kappa_p\epsilon^k}\mathcal{A}^*(b + R_1^k) + C - R_2^k\right\}$ is bounded and does not depend on $\{X^k\}$. Substituting this expression into the third equality yields that 
	\[
	(1+\kappa_c\epsilon^k)X^k - \Phi\left(\epsilon^k, X^k - \frac{1}{\kappa_p \epsilon^k}\mathcal{A}^*\mathcal{A}X^k - C^k\right) = R_3^k.
	\]
	Recall that $\Phi$ is a globally Lipschitz continuous mapping and $\Phi(0,\cdot) = \Pi_{\mathbb{S}_+^n}(\cdot)$. We see that there exists a bounded sequence $\{R_4^k\}$ such that
	\[
	\Phi\left(\epsilon^k, X^k - \frac{1}{\kappa_p \epsilon^k}\mathcal{A}^*\mathcal{A}X^k - C^k\right) =  \Pi_{\mathbb{S}_+^n}\left( X^k - \frac{1}{\kappa_p \epsilon^k}\mathcal{A}^*\mathcal{A}X^k\right) + R_4^k,\quad \forall k\geq 0.
	\]
	Consequently, it holds that
	\[
	X^k - \Pi_{\mathbb{S}_+^n}\left(\frac{1}{1+\kappa_c\epsilon^k}X^k - \frac{1}{\kappa_p \epsilon^k(1+\kappa_c\epsilon^k)}\mathcal{A}^*\mathcal{A}X^k\right) = R_5^k,
	\]
	where $\{R_5^k:=(1+\kappa_c\epsilon^k)^{-1} (R_3^k - R_4^k)\}$ is also bounded. Notice that $\kappa_p>0,\kappa_c>0$ and for any $\epsilon\in [\epsilon', \epsilon'']$, it holds that
	\[
	\begin{aligned}
		&\; \frac{1}{1+\kappa_c\epsilon}X - \frac{1}{\kappa_p \epsilon(1+\kappa_c\epsilon)}\mathcal{A}^*\mathcal{A}X \\
		= &\; X - \frac{1}{\kappa_p \epsilon''(1+\kappa_c\epsilon'')}\mathcal{A}^*\mathcal{A}X - \frac{\kappa_c\epsilon}{1+\kappa_c\epsilon}X - \left(\frac{1}{\kappa_p \epsilon(1+\kappa_c\epsilon)}-\frac{1}{\kappa_p \epsilon''(1+\kappa_c\epsilon'')}\right)\mathcal{A}\mathcal{A}^*X. 
	\end{aligned}
	\]
	One can derive from \cite[Theorems 2.8 \& 4.4]{sun2001solving} that, for any positive constant $c$ the set 
	\[
	\left\{X\in\mathbb{S}^n\;:\; \norm{X - \Pi_{\mathbb{S}_+^n}\left(\frac{1}{1+\kappa_c\epsilon}X - \frac{1}{\kappa_p \epsilon(1+\kappa_c\epsilon)}\mathcal{A}^*\mathcal{A}X\right)}\leq c, \epsilon\in [\epsilon', \epsilon'']\right\}
	\]
	is bounded. Since there exists a sufficiently large $c> 0$ such that the sequence $\{X^k\}$ remains within the associated bounded set, we conclude that the sequence $\{X^k\}$ is bounded, which in turn implies that $\{(\epsilon^k, X^k, y^k, Z^k)\}$ is also bounded. In this case, an accumulation point must exist. 
	
	Let $(\bar \epsilon, \bar X, \bar y, \bar Z)$ be any accumulation point (if it exists) of $ \{(\epsilon^k, X^k, y^k, Z^k)\} $. By taking a subsequence if necessary, we may assume that $\{(\epsilon^k, X^k, y^k, Z^k)\}$ converges to $(\bar \epsilon, \bar X, \bar y, \bar Z)$. Then, by the continuity of $\psi(\cdot)$, it holds that
	\[
	\bar\psi = \psi(\bar \epsilon, \bar X, \bar y, \bar Z), \quad \bar \zeta = \zeta(\bar \epsilon, \bar X, \bar y, \bar Z),\quad (\bar \epsilon, \bar X, \bar y, \bar Z) \in \cN.
	\]
	By the fact that $ (\bar \epsilon, \bar X, \bar y, \bar Z) \in \cN $, we see that $\bar \epsilon \in \R_{++}$. Therefore, by Lemma \ref{lemma-nonsingular}, we see that there {exists} a neighborhood of $(\bar \epsilon, \bar X, \bar y, \bar Z)$, denoted by $\cU$, such that $\hcE'(\epsilon, X, y, Z)$ is nonsingular for any $(\epsilon, X, y, Z) \in \cU$ with $\epsilon > 0$. {Note that for $ k $ sufficiently large, we have that $ (\epsilon^k, X^k, y^k, Z^k) $ belongs to $\cU$ with $\epsilon^k > 0$. Furthermore, by Lemma \ref{lemma-step-size}, there also exists $\bar\alpha \in (0, 1]$ such that
	for any $\alpha\in (0, \bar \alpha]$,
	\[
	\psi(\epsilon^k + \alpha \Delta \epsilon^k, X^k + \alpha \Delta X^k, y^k + \alpha \Delta y^k, Z^k + \alpha \Delta Z^k) \leq [1 - 2\sigma(1-\delta)\alpha]\psi(\epsilon^k, X^k, y^k, Z^k)\,
	\]
	for $k\geq 0$ sufficiently large. The existence of the fixed number $\bar \alpha\in (0,1]$ further indicates that} there exists a nonnegative integer $\ell$ such that $\rho^\ell\in (0, \bar \alpha]$ and  $\rho^{\ell_k} \geq \rho^\ell$ for all $k$ sufficiently large. Therefore, it holds that
	\[
		\psi(\epsilon^{k+1}, X^{k+1}, y^{k+1}, Z^{k+1}) \leq [1-2\sigma(1-\delta)\rho^{\ell_k}] \psi(\epsilon^k, X^k, y^k, Z^k) \leq [1-2\sigma(1-\delta)\rho^{\ell}] \psi(\epsilon^k, X^k, y^k, Z^k),
	\]
	for all sufficiently large $k$. The above inequality implies that $\bar \psi \leq 0$, which is a contradiction {to the assumption $\bar\phi > 0$}. Hence, $\bar\psi = 0$, which implies that $\hcE(\bar \epsilon, \bar X, \bar y, \bar Z) = 0$. 
	
	Last, we assume that the solution set to the KKT system is nonempty and bounded. One can check that $\hcE(\cdot)$ is a weakly univalent function \cite{gowda1999weak}. Then, from \cite{ravindran2001regularization}, one can verify that $\hcE^{-1}(0)$ is also nonempty and bounded, which further implies the boundedness of the sequence $\{(\epsilon^k, X^k, y^k, Z^k)\}$. Thus, the proof is completed. 

\section{Proof of Lemma \ref{lemma-delta-ineq}} \label{proof-delta-ineq}
Suppose that $ \bar W $ has the spectral decomposition \eqref{eq-eig-W}. Denote $ \Delta \tilde X :=  P^T\Delta X P$ and $ \Delta \tilde Z = P^T\Delta Z P$. For the index set $ \beta $, define $ \Phi_{\abs{\beta}} $ as before. Then, by Lemma \ref{lemma-partial-B-Phi}, there exists $ V_{\abs{\beta}}\in \partial \Phi_{\abs{\beta}}(0,0) $ such that
	\[
		V(0, \Delta X - \Delta Z) = P
		\begin{pmatrix}
			\Delta \tilde H_{\alpha\alpha} & \Delta \tilde H_{\alpha\beta} & [\Omega_0(d)]_{\alpha\gamma}\circ \Delta \tilde H_{\alpha\gamma} \\
			\Delta \tilde H_{\alpha\beta}^T & V_{\abs{\beta}}(0, \Delta \tilde H_{\beta\beta}) & 0 \\
			\Delta \tilde H_{\alpha\gamma}^T\circ [\Omega_0(d)]_{\alpha\gamma}^T & 0 & 0
		\end{pmatrix} P^T,
	\]
	where $ \Delta \tilde H = \Delta \tilde X - \Delta \tilde Z $. Comparing both sides of the relation $ \Delta X = V(0, \Delta X - \Delta Z) $ yields that
	\[
		\Delta \tilde Z_{\alpha\alpha} = 0, \quad \Delta \tilde Z_{\alpha\beta}  = 0,\quad \Delta \tilde X_{\beta\gamma} = 0,\quad \Delta \tilde X_{\gamma\gamma} = 0,
	\]
	and that
	\[
		\Delta \tilde X_{\beta\beta} = V_{\abs{\beta}}(0,\Delta \tilde X_{\beta\beta} - \Delta\tilde Z_{\beta\beta}),\; \Delta \tilde X_{\alpha\gamma} - [\Omega_0(d)]_{\alpha\gamma}\circ \Delta \tilde X_{\alpha\gamma} = [\Omega_0(d)]_{\alpha\gamma}\circ \Delta \tilde Z_{\alpha\gamma}.
	\]
	By Lemma \ref{lemma-V-PSD}, it holds that
	\[
		\inprod{\Delta \tilde X_{\beta\beta}}{-\Delta \tilde Z_{\beta\beta}} = \inprod{V_{\abs{\beta}}(0,\Delta \tilde X_{\beta\beta} - \Delta\tilde Z_{\beta\beta})}{\left(\Delta \tilde X_{\beta\beta} - \Delta\tilde Z_{\beta\beta}\right) - V_{\abs{\beta}}(0,\Delta \tilde X_{\beta\beta} - \Delta\tilde Z_{\beta\beta})}\geq 0
	\]
	which implies that
	\[
		\inprod{\Delta X}{\Delta Z} =  \inprod{\tilde \Delta X}{\tilde \Delta Z} = \inprod{\Delta \tilde X_{\beta\beta}}{\Delta \tilde Z_{\beta\beta}} + 2\inprod{\Delta \tilde X_{\alpha\gamma}}{\Delta \tilde Z_{\alpha\gamma}}
			 \leq  2\inprod{\Delta \tilde X_{\alpha\gamma}}{\Delta \tilde Z_{\alpha\gamma}}.
	\]
	However, simple {calculations show} that $ \Gamma_{\bar X}(-\bar Z, \Delta X) =  2\inprod{\Delta \tilde X_{\alpha\gamma}}{\Delta \tilde Z_{\alpha\gamma}}$. This proves the lemma. 

\section{Proof of Proposition \ref{prop-nonsingular-cE}} \label{proof-nonsingular-cE}
It is obvious that part 3 implies part 2. From \eqref{eq-V0-VH}, we see that for any $V_0\in \partial_B\Pi_{\S_+^n}(W)$, $W\in \S^n$, there exists $V\in \partial_B \Phi(0, W)$ such that
	\[
	V_0(H) = V(0, H),\quad \forall H\in \S^n.
	\]
	Therefore, by Theorem \ref{thm-pd-nonsingular}, part 2 implies part 1. So, it suffices to show that part 1 implies part 3.
	
	Now, suppose that part 1 holds. {Recall that} the primal constraint nondegeneracy condition at $\bar X$ implies that $\cM(\bar X) = \left\{(\bar y, \bar Z)\right\}$, {i.e., the dual multiplier is unique}. Now, by using Lemma \ref{lemma-sosc-dual-non}, we can see that the dual constraint nondegeneracy at $(\bar y, \bar Z)$ implies that the strong {second-order sufficient condition} holds at $\bar X$. In particular, it holds that
	\begin{equation}\label{eq-thm-pd-nonsingular-1}
		-\Gamma_{\bar X}(-\bar Z, H) > 0,\quad \forall 0\neq H \in \mathrm{app}(\bar y, \bar Z).
	\end{equation}
	
	Let {$U\in \partial \hcE(\bar\epsilon, \bar X, \bar y, \bar Z)$} (recall that $ \bar \epsilon = 0 $). Then there exists {$V\in \partial \Phi(0, \bar X - \bar Z)$} such that
	\[
	U(\Delta \epsilon, \Delta X, \Delta y, \Delta Z) =
	\begin{pmatrix}
		\Delta \epsilon \\
		\cA \Delta X \\
		-\cA^*\Delta y - \Delta Z \\
		\Delta X - V(\Delta \epsilon, \Delta X - \Delta Z)
	\end{pmatrix},\quad (\Delta \epsilon, \Delta X, \Delta y, \Delta Z)\in \R\times \X.
	\]
	To show that $U$ is nonsingular, it suffices to show that $U(\Delta \epsilon, \Delta X, \Delta y, \Delta Z) = 0$ implies that $(\Delta \epsilon, \Delta X, \Delta y, \Delta Z) = 0$. So, let us assume that $U(\Delta \epsilon, \Delta X, \Delta y, \Delta Z) = 0$, i.e.,
	\begin{equation}\label{eq-thm-pd-nonsingular-2}
		\begin{pmatrix}
			\Delta \epsilon \\
			\cA \Delta X \\
			-\cA^*\Delta y - \Delta Z \\
			\Delta X - V(\Delta \epsilon, \Delta X - \Delta Z)
		\end{pmatrix} = 0.
	\end{equation}
	
	We first prove that $\Delta X = 0$ by contradiction. To this end,  we assume that $\Delta X \neq 0$. By Lemma \ref{lemma-partial-B-Phi} and \eqref{eq-thm-pd-nonsingular-2}, one can verify that
	\begin{equation}\label{eq-thm-pd-nonsingular-3}
		\Delta X \in \mathrm{app}(\bar y, \bar Z).
	\end{equation}
	Then by \eqref{eq-thm-pd-nonsingular-1}, \eqref{eq-thm-pd-nonsingular-3} implies that
	\begin{equation}\label{eq-thm-pd-nonsingular-4}
		-\Gamma_{\bar X}(-\bar Z, \Delta X) > 0.
	\end{equation}
	On the other hand, by \eqref{eq-thm-pd-nonsingular-2}, it holds that $\inprod{\Delta X}{\Delta Z} = \inprod{\Delta X}{-\cA^*\Delta y} = \inprod{\cA\Delta X}{-\Delta y} = 0$, which together with Lemma \ref{lemma-delta-ineq}, yields that
	\begin{equation}\label{eq-thm-pd-nonsingular-5}
		\Gamma_{\bar X}(-\bar Z, \Delta X) \geq \inprod{\Delta X}{\Delta Z} = 0.
	\end{equation}
	However, \eqref{eq-thm-pd-nonsingular-5} contradicts to \eqref{eq-thm-pd-nonsingular-4}. Thus, $\Delta X = 0$ and $V(0, -\Delta Z) = 0$.
	
	We next prove $\Delta y = 0$ and $\Delta Z = 0$. From Lemma \ref{lemma-partial-B-Phi}, $ V(0, -\Delta Z) = 0 $ implies that that
	\[
	P_\alpha^T\Delta Z P_\alpha  = 0,\quad P_\alpha^T\Delta Z P_\beta = 0, \quad P_\alpha^T\Delta Z P_\gamma = 0.
	\]
	On the other hand, from \eqref{eq-thm-pd-nonsingular-2}, we get $\cA^*\Delta y + \Delta Z = 0$. Since the primal constraint {nondegeneracy} condition \eqref{eq-def-primal-non} holds at $\bar X$, there exist $X\in \S^n$ and $Z\in \mathrm{lin}(\cT_{\S^n_+}(\bar X))$ such that $\cA X = \Delta y$, and $ X + Z = \Delta Z$. As a consequence, it holds that
	\[
	\begin{aligned}
		\hspace{-4mm}
		 \inprod{\Delta y}{\Delta y} + \inprod{\Delta Z}{\Delta Z} \;= & \; \inprod{\cA X}{\Delta y} + \inprod{X+Z}{\Delta Z} = \inprod{\cA X}{\Delta y} + \inprod{X}{-\cA^*(\Delta y)} + \inprod{Z}{\Delta Z} \\
		= &\; \inprod{Z}{\Delta Z}
		\;= \; \inprod{P^TZP}{P^T\Delta Z P}
		\;= \; 0,
	\end{aligned}
	\]
	where we have used the {fact} that
	$
	P_\alpha^T\Delta Z P_\alpha  = 0,\; P_\alpha^T\Delta Z P_\beta = 0, \; P_\alpha^T\Delta Z P_\gamma = 0,\;Z\in \mathrm{lin}(\cT_{\S^n_+}(\bar X)).
	$
	Thus, $\Delta y = 0$, $\Delta Z = 0$ and hence $U$ is nonsingular. Thus, the proof is completed.
\section{Proof of Theorem \ref{thm-superlinear-convergence}} \label{proof-superlinear-convergence}
By Theorem \ref{thm-global-convergence}, we see that $\hcE(\bar \epsilon, \bar X, \bar y, \bar Z) = 0$ and in particular, $\bar \epsilon = 0$. Then by Proposition \ref{prop-nonsingular-cE}, the primal and dual constraint nondegeneracy conditions imply that every element of $\partial \hcE(0, \bar X, \bar y, \bar Z)$ (hence, of $\partial_B \hcE(0, \bar X, \bar y, \bar Z)$) is nonsingular. As a consequence (see e.g., \cite[Proposition 3.1]{qi1993nonsmooth}), for all $k$ sufficiently large, it holds that
	\begin{equation}
	\label{eq-thm-inv-O1}
	\norm{\hcE'(\epsilon^k, X^k, y^k, Z^k)^{-1}} = O(1).	
	\end{equation}
	For simplicity, in this proof, we denote  $\bar w:= (\bar \epsilon, \bar X, \bar y, \bar Z)^T$, $w^k:= (\epsilon^k, X^k, y^k, Z^k)^T$, $\Delta w^k:= (\Delta \epsilon^k, \Delta X^k, \Delta y^k, \Delta Z^k)^T$,  $\hcE_k:= \hcE(\epsilon^k, X^k, y^k, Z^k)$, $\hcE_*:= \hcE(\bar \epsilon, \bar X, \bar y, \bar Z)$ and $\cJ_k:= \hcE'(\epsilon^k, X^k, y^k, Z^k)$ for $k\geq 0$. Then, we can verify that 
		\begin{align}\label{eq-thm-local-1}
				 \norm{w^k + \Delta w^k - \bar w} = &\;\norm{w^k + \cJ_k^{-1}\left(\begin{pmatrix}
					\zeta_k\hat \epsilon \\ \cR_k
				\end{pmatrix} - \hcE_k\right) - \bar w}\notag = \norm{-\cJ_k^{-1}\left( \hcE_k - \cJ_k (w^k - \bar w)  - \begin{pmatrix}
					\zeta_k\hat \epsilon \\ \cR_k
			\end{pmatrix} \right)} \notag\\
			= &\; O\left( \norm{\hcE_k - \cJ_k (w^k - \bar w)} \right) + O\left(\norm{\hcE_k}^{1+\tau}\right) + O(\norm{\cR_k}).
		\end{align}
	Since $\hcE$ is locally Lipschitz continuous at $(\bar\epsilon, \bar X, \bar y, \bar Z)$, for all $k$ sufficiently large, it holds that
	\begin{equation}\label{eq-thm-local-2}
		\norm{\hcE_k}=\norm{\hcE_k - \hcE_*} = {O\left( \norm{w^k - \bar w } \right)}.
	\end{equation}
	Moreover, since $\cE'_\epsilon$ is bounded near $(\bar\epsilon, \bar X, \bar y, \bar Z)$, it holds that
	\begin{equation}\label{eq-thm-local-3}
		\begin{aligned}
			\norm{\cR_k} \leq &\; \eta_k \norm{\cE(\epsilon^k, X^k, y^k, Z^k) + \cE_\epsilon'(\epsilon^k, X^k, y^k, Z^k)\Delta \epsilon^k} \\
			\leq &\; O\left(\norm{\hcE_k}^{\tau}\right)\left( \norm{\cE(\epsilon^k, X^k, y^k, Z^k)} + O\left(\abs{\Delta \epsilon^k}\right) \right) \\
			\leq &\; O\left(\norm{\hcE_k}^{\tau}\right)\left( \norm{\cE(\epsilon^k, X^k, y^k, Z^k)} + O\left(\abs{-\epsilon^k + \zeta_k \hat \epsilon}\right) \right) \\
			\leq &\; O\left(\norm{\hcE_k}^{1+\tau}\right) = O\left(\norm{\hcE_k - \hcE_*}^{1+\tau}\right) =  {O\left( \norm{w^k - \bar w}^{1+\tau} \right)}.
		\end{aligned}
	\end{equation}
	Since $\Phi$ is strongly semismooth everywhere {(see Proposition \ref{prop-huber-properties})}, $\hcE$ is also strongly semismooth at $(\bar\epsilon, \bar X, \bar y, \bar Z)$. Thus, for $k$ sufficiently large, it holds that
	\begin{equation*}
		\begin{aligned}
			&\; \norm{\hcE_k - \cJ_k (w^k - \bar w)} =  {O\left( \norm{w^k - \bar w}^2 \right)},
		\end{aligned}
	\end{equation*}
	which together with \eqref{eq-thm-local-1}--\eqref{eq-thm-local-3} implies that
	\begin{equation}
		\label{eq-thm-local-4}
		\norm{w^k + \Delta w^k - \bar w} 
			= {O\left(\norm{w^k - \bar w}^{1+\tau}\right)}.
	\end{equation}
	Now by using the strong semismoothness of $\hcE$ again and \eqref{eq-thm-inv-O1}, we can show that for $k$ sufficiently large,
	\begin{equation}\label{eq-thm-local-5}
		\norm{w^k - \bar w} = O\left(	\norm{\hcE_k} \right).
	\end{equation}
	Combining \eqref{eq-thm-local-4}--\eqref{eq-thm-local-5} and the fact that $\hcE$ (hence $\psi$) is locally Lipschitz continuous at $(\bar\epsilon, \bar X, \bar y, \bar Z)$, we get for $k$ sufficiently large that
	\[
	\begin{aligned}
		&\; \psi(\epsilon^{k} + \Delta \epsilon ^k , X^k + \Delta X^k,  y^{k} +
		\Delta y^k, Z^{k} + \Delta Z^k) \\
		= &\; \norm{\hcE(\epsilon^{k} + \Delta \epsilon ^k , X^k + \Delta X^k, y^{k} +
			\Delta y^k, W^{k} + \Delta W^k) - \hcE_*}^2\\
		= &\;  O\left( 	\norm{w^k + \Delta w^k - \bar w}^2\right) = O\left(\norm{w^k - \bar w}^{2(1+\tau)}\right)  = O\left(\norm{\hcE_k}^{2(1+\tau)}\right) \\
		= &\;  O\left(\norm{\psi(\epsilon^k, X^k, y^k, Z^k)}^{1+\tau} \right) = o\left(\norm{\psi(\epsilon^k, X^k, y^k, Z^k)}\right).
	\end{aligned}
	\]
	This shows that, for $k$ sufficiently large, $w^{k+1} = w^k + \Delta w^k$, 
	i.e., the unit step size is eventually {accepted.} Therefore, the proof is completed. 

\bibliographystyle{plain}
\bibliography{references.bib}

\end{document}